\documentclass[a4paper,10pt,draft]{amsart}
\usepackage{amsmath,amsthm,amssymb,amsfonts,enumerate,color}

\newcommand{\norm}[1]{\left\Vert#1\right\Vert}
\newcommand{\abs}[1]{\left\vert#1\right\vert}
\newcommand{\set}[1]{\left\{#1\right\}}
\newcommand{\Real}{\mathbb{R}}
\newcommand{\supp}{\operatorname{supp}}

\renewcommand{\L}{\mathcal{L}}
\renewcommand{\P}{\mathcal{P}}
\newcommand{\T}{\mathcal{T}}

\newtheorem{thm}{Theorem}[section]
\newtheorem{prop}[thm]{Proposition}

\newtheorem{lem}[thm]{Lemma}

\theoremstyle{definition}
\newtheorem{defn}[thm]{Definition}
\newtheorem{rem}[thm]{Remark}

\numberwithin{equation}{section}

\author[T. Ma]{Tao Ma}
\address{School of Mathematics and Statistics \\
Wuhan University\\
430072 Wuhan, China}
\email{tma.math@whu.edu.cn}

\author[P. R. Stinga]{Pablo Ra\'ul Stinga}
\address{Departamento de Matem\'aticas y Computaci\'on\\
Universidad de La Rioja\\
26006 Logro\~no, Spain}
\email{pablo-raul.stinga@unirioja.es}

\author[J. L. Torrea]{Jos\'e L. Torrea}
\address{Departamento de Matem\'aticas\\
Universidad Aut\'onoma de Madrid\\
28049 Madrid, Spain  and ICMAT-CSIC-UAM-UCM-UC3M}
\email{joseluis.torrea@uam.es}

\author[C. Zhang]{Chao Zhang}
\address{School of Mathematics and Statistics \\
          Wuhan University \\
          430072 Wuhan, China}
\address{\textit{Current address:}
          \vskip 0.04cm  Departamento de Matem\'aticas \\
          Universidad Au\-t\'o\-no\-ma de Madrid \\
          28049 Madrid, Spain}
\email{zaoyangzhangchao@163.com}

\thanks{Research partially supported by Ministerio de Ciencia e Innovaci\'{o}n de Espa\~{n}a MTM2008-06621-C02-01.
The first and fourth authors were partially supported by National Natural Science Foundation of China No.11071190.
Second author was also supported by grant COLABORA 2010/01 from Planes Riojanos de I+D+I}



\begin{document}

\title{Regularity properties of Schr\"odinger operators}

\begin{abstract}
Let $\L$ be a Schr\"odinger operator of the form $\L=-\Delta+V$, where the nonnegative potential $V$ satisfies a reverse H\"older inequality. Using the method of $\L$-harmonic extensions we study regularity estimates at the scale of adapted H\"older spaces. We give a pointwise description of $\L$-H\"older spaces and provide some characterizations in terms of the growth of fractional derivatives of any order and Carleson measures. Applications to fractional powers of $\L$ and multipliers of Laplace transform type developed.
\end{abstract}

\maketitle

\section{Introduction}

One of the methods applied to develop regularity estimates in the theory of partial
 differential equations is to consider equivalent formulations of the problems by adding
 a new variable. Let us give a rough description of the idea. Suppose that we want to study
 regularity properties of a certain function $f(x)$ defined in some domain $\Omega$.
  Take $f$ as the Dirichlet or initial data for some PDE $Au=0$ in the variables $x\in\Omega$
  and $t$ in an interval $I$. The question is the following: which properties of the
  solution $u$ in $\Omega\times I$ imply regularity of $f$, the boundary data?
  The most simple and classical situation to consider is the following:
\begin{equation}\label{Dir prob}
\left\{
  \begin{array}{ll}
    Au\equiv\partial_{tt}u+\Delta u=0, & \hbox{in}~\Real^n\times(0,\infty), \\
    u(x,0)=f(x), & \hbox{on}~\Real^n.
  \end{array}
\right.
\end{equation}
Here $\Delta$ is the Laplacian in $\Real^n$. Then $u$ is the harmonic extension of $f$, namely
\begin{equation}\label{Classical Poisson}
u(x,t)=e^{-t(-\Delta)^{1/2}}f(x).
\end{equation}
Note that we have $-u_t(x,0)=(-\Delta)^{1/2}f(x)$. Therefore, the harmonic extension $u$ can give some information not only about $f$ but also about the fractional Laplacian, a nonlocal operator, acting on $f$. It is worth to mention here that such a remarkable fact was applied to show that weak solutions of the critical dissipative quasi-geostrophic equation are H\"older continuous, see \cite{Caffarelli-Vasseur}.

In general, to study the regularity properties of fractional
operators like $(-\Delta)^{1/2}$, or more generally
$(-\Delta)^{\sigma/2}$ and $(-\Delta)^{-\sigma/2}$, $0<\sigma<2$,
there are essentially two possible alternatives. Either describe the
operators with a pointwise integro-differential or integral formula,
or characterize the H\"older classes by some norm estimate of
harmonic extensions \eqref{Dir prob}, that are in fact Poisson
integrals \eqref{Classical Poisson}. The first approach was taken by
L. Silvestre in \cite{Silvestre} to analyze how
$(-\Delta)^{\pm\sigma/2}$ acts on the H\"older spaces
{$C^{0,\alpha}$}. Let us point out that he also needed to handle the
Riesz transforms $\partial_{x_i}(-\Delta)^{-1/2}$ as operators on
{$C^{0,\alpha}$}. The second one, in the spirit of harmonic
extensions, is nowadays classical. Indeed, for bounded functions $f$
it is well known that the harmonic extension \eqref{Classical
Poisson} satisfies $\norm{tu_t(\cdot,t)}_{L^\infty(\Real^n)}\leq
Ct^\alpha$ for all $t>0$ if, and only if, $f\in {C^{0,\alpha}}$,
$0<\alpha<1$, see for instance \cite{SteinSingular}.

In this paper we consider the time independent Schr\"odinger operator in $\Real^n$, $n\geq3$,
\begin{equation}\label{L}
\L:=-\Delta+V,
\end{equation}
where the nonnegative potential $V$ satisfies a reverse H\"older
inequality for some $q>n/2$, see \eqref{RHs} below. Observe that the
reverse H\"older condition is just an integrability property, so no
smoothness on $V$ is assumed.  Our aim is to develop the regularity
theory of H\"older spaces adapted to $\L$ and to study estimates of
operators like {\bf fractional integrals} $\L^{-\sigma/2}$, and {\bf
fractional powers} $\L^{\sigma/2}$. Such operators can be defined by
using $\L$--harmonic extensions. The solution of the boundary value
problem
\begin{equation}\label{L-harmonic extension}
\left\{
  \begin{array}{ll}
    \partial_{tt}u-\L u=0, & \hbox{in}~\Real^n\times(0,\infty), \\
    u(x,0)=f(x), & \hbox{on}~\Real^n,
  \end{array}
\right.
\end{equation}
is given by the action of the $\L$--Poisson semigroup on $f$:
$$u(x,t)=\P_tf(x)\equiv e^{-t\sqrt{\L}}f(x).$$
Let us recall that Bochner's subordination formula gives a way to express $u$ as a mean in the time variable of the solution of the $\L$--diffusion equation, see \eqref{subordinacion}. The powers of $\L$ can be described in terms of $u$ as in \eqref{fi} and \eqref{fl}. Therefore, to deal with spaces and operators, we will adopt the point of view based on $\L$--harmonic extensions \eqref{L-harmonic extension}.

Our choice of the method turns out to be well suited for our purposes. In this Schr\"odinger context the pointwise description of the operators as in \cite{Silvestre} seems to be technically difficult. In fact, even for one of the most simplest cases (the harmonic oscillator, where $V(x)=\abs{x}^2$) it is already rather involved, see \cite{Stinga-Torrea}. On the other hand, the characterization of $\L$--H\"older spaces via $\L$--harmonic extensions does not appear to be easily obtained as a repetition of the arguments for classical H\"older spaces given in \cite{SteinSingular}.

Let us begin with the definition of H\"older spaces naturally associated to $\L$. The concept is based on the \textit{critical radii function} $\rho(x)$ defined by Z. Shen in \cite{Shen}, see \eqref{critical}.

\begin{defn}[H\"older spaces for $\L$]\label{defC}
A continuous function $f$ defined on $\Real^n$ belongs to the space
$C^{0,\alpha}_\L$, $0<\alpha\leq1$, if the quantities
$$[f]_{C^{\alpha}}=\sup_{\begin{subarray}{c}x,y\in\Real^n\\x\neq y\end{subarray}}
\frac{\abs{f(x)-f(y)}}{\abs{x-y}^\alpha}\quad\hbox{and}\quad[f]_{M^\alpha_\L}=\sup_{x\in\Real^n}\abs{\rho(x)^{-\alpha} f(x)},$$
are finite. The norm in the spaces $C_\L^{0,\alpha}$ is $\norm{f}_{C_\L^{0,\alpha}}=[f]_{C^{\alpha}}+[f]_{M^\alpha_\L}$.
\end{defn}

The first main theorem of the paper is the following regularity result.

\begin{thm}\label{Thm:Operators}
Assume that $q>n$. Let $\sigma$ be a positive number, $0<\alpha<1$
and $f\in C^{0,\alpha}_\L$.
\begin{enumerate}[(a)]
    \item If $0<\alpha+\sigma<1$ then $\L^{-\sigma/2}f\in C^{0,\alpha+\sigma}_\L$ and $\|\L^{-\sigma/2}f\|_{C^{0,\alpha+\sigma}_\L}\leq C\norm{f}_{C^{0,\alpha}_\L}$.
    \item If $\sigma<\alpha$ then $\L^{\sigma/2}f\in C_\L^{0,\alpha-\sigma}$ and $\|\L^{\sigma/2}f\|_{C^{0, \alpha-\sigma}_\L}\leq C\norm{f}_{C^{0,\alpha}_\L}$.
    \item Let $a$ be a bounded function on $[0,\infty)$ and define
        $$m(\lambda)=\lambda^{1/2}\int_0^\infty e^{-s\lambda^{1/2}}a(s)~ds,\quad \lambda>0.$$
        Then the multiplier operator of Laplace transform type $m(\L)$ is bounded on $C^{0,\alpha}_\L$, $0< \alpha<1$.
\end{enumerate}
\end{thm}

In order to prove Theorem \ref{Thm:Operators} we shall need a characterization of functions $f$ in $C^{0,\alpha}_\L$
 by means of size and integrability properties of $\L$--harmonic extensions \eqref{L-harmonic extension} to
 the upper half space. The theory of $BMO_\L$ spaces and Carleson measures developed in \cite{DGMTZ} will
  be a central tool. In fact our result provides a
  characterization of the $\L$--H\"older classes via Carleson measures.
   Moreover, our statement not only involves first order derivatives of
   the $\L$--Poisson semigroup but {also} introduces higher and fractional order
    derivatives. The concept of fractional derivative that we give here is
    of independent interest and allows us to present a more general characterization.
     Given a positive number $\beta$, let us denote by $m$ the smallest integer which
     strictly exceeds $\beta$, that is, $[\beta]+1$. Let $F(x,t)$ be a reasonable nice function of $x\in\Real^n$ and $t>0$. We define, following C. Segovia and R. L. Wheeden \cite{Segovia-Wheeden},
\begin{equation}\label{frac deriv}
\partial_t^\beta F(x,t)=\frac{e^{-i \pi(m-\beta)}}{\Gamma(m-\beta)}
\int_0^\infty{\partial_t^mF}(x,t+r)r^{m-\beta}~\frac{dr}{r},\quad
x\in\Real^n,~t>0.
\end{equation}
Note that in the definition above $\partial_t^1=\partial_t$. The following is the second main result.

\begin{thm}\label{Thm:Characterization}
Let $0<\alpha<1$ and $f$ be a function such that
$f(x)(1+\abs{x})^{-(n+\alpha+\varepsilon)}$ is integrable for any
$\varepsilon>0$. Fix any $\beta>\alpha$ and assume that $q>n$. The
following statements are equivalent:
\begin{enumerate}[(i)]
    \item $f\in C^{0,\alpha}_\L$.
    \item There exists a constant $c_{1,\beta}$ such that $\|t^\beta\partial_t^\beta\P_tf\|_{L^\infty(\mathbb{R}^n)}\leq c_{1,\beta}t^\alpha$.
    \item There exists a constant $c_{2,\beta}$ such that for all balls $B=B(x_0,r)$ in $\Real^n$,
     $$\left(\frac{1}{\abs{B}}\int_{\widehat{B}}|t^\beta\partial_t^\beta\P_tf(x)|^2~\frac{dx~dt}{t}\right)^{1/2}
     \leq c_{2,\beta}\abs{B}^{\frac{\alpha}{n}},$$
    where $\widehat{B}$ denotes the tent over $B$ defined by $\set{(x,t): x\in B, and ~{0< t\le r}}$.
\end{enumerate}
Moreover, the constants $c_{1,\beta}$, $c_{2,\beta}$ and $\norm{f}_{{C^{0,\alpha}_\L}}$ above are comparable.
\end{thm}

Some observations are in order. The integrability condition required
on $f$ in Theorem \ref{Thm:Characterization} implies that the
$\L$--harmonic extension $\P_tf$ is well defined, see Proposition
\ref{Prop:Poisson est}\textit{(a)} below.
 Such a condition is weaker than to ask for $f$ to be bounded (as in the classical case, see \cite{SteinSingular}) or even to
  have the growth $\abs{f(x)}\leq C\rho(x)^\alpha$ that appears in the
definition of $\L$--H\"older space above, see Lemma
\ref{crucial}\textit{(i)}. The Carleson property \textit{(iii)} can
be proved since there is an available  Campanato-type description of
$C^{0,\alpha}_\L$. This identification was proved by Bongioanni,
Harboure and Salinas in \cite{Bongioanni-Harboure-Salinas-weighted},
see Proposition \ref{Prop:BMO y Calpha}.

Under the light of Definition \ref{defC} and Theorem \ref{Thm:Characterization},
the natural question is how to define and characterize higher-order $\L$--H\"older spaces,
that is, spaces of the type $C^{k,\alpha}_\L$ for $k$ a positive integer. It is already known the
 characterization of classical $C^{k,\alpha}$ spaces by size properties of harmonic extensions,
 see \cite{SteinSingular}. In the case of the harmonic oscillator $H=-\Delta+\abs{x}^2$,
  the definition of the H\"older spaces $C^{k,\alpha}_H$ was given in \cite{Stinga-Torrea}.
In the case of general potentials $V$, because of the lack of
smoothness we will not try to consider higher-order $\L$--H\"older
spaces. Nevertheless, as it happens in the classical case
\cite{SteinSingular}, we could define higher-order spaces by using
property \textit{(ii)} of Theorem \ref{Thm:Characterization} in the
following way. Let $\alpha>0$ and fix any $\beta>\alpha$. Then we
would say that a function $f$ belongs to the $\L$--H\"older space
$\Lambda^\alpha_\L$ if
$\|t^\beta\partial_t^\beta\P_tf\|_{L^\infty(\Real^n)}\leq
Ct^\alpha$. Note that this new concept depends on the choice of
$\beta$, but in fact we can show that it does not, see Lemma
\ref{Lem:equivalence} below. If $0<\alpha<1$ then the definition
agrees with Definition \ref{defC}. But when $\alpha>1$ and $V$ is
not smooth it is not clear how to give an equivalent pointwise
formulation to measure the smoothness of $f$ as in the classical
way. For the potential $V=\abs{x}^2$ some results in this direction
can be obtained and they will appear in a forthcoming work.

The condition $q>n$ in Theorem \ref{Thm:Characterization} seems to
be natural if we expect to have some regularity for the operators
involved. See Z. Shen \cite{Shen} for a discussion in $L^p$ and
\cite{Bongioanni-Harboure-Salinas-Riesz} in the $BMO^\alpha_\L$
context.

We also consider the extreme values of $\alpha$. Note that the conclusion of Theorem \ref{Thm:Characterization} above is not valid in the cases $\alpha=1$ or $\alpha=0$. In fact, we have the following results for $\alpha=1$:

\begin{thm}[Case $\alpha=1$]\label{Thm:alpha 1}
Assume that $q>n$.
\begin{enumerate}[(I)]
\item  If $f\in C^{0,1}_\L$ then for any
$\beta>1$ there exists a constant $c_\beta$ such that
$$\left(\frac{1}{\abs{B}}\int_{\widehat{B}}|t^\beta\partial_t^\beta\P_tf(x)|^2~\frac{dx~dt}{t}\right)^{1/2}
\leq c_\beta\abs{B}^{\frac{1}{n}},$$ for all balls $B$. The converse statement is not true.
\item Let $\L_\mu=-\Delta+\mu$, for $\mu>0$. There exists a function $f$ such that for any $\beta>1$ there exists a
 constant $c_\beta$ that verifies $\|t^\beta \partial^\beta_t \P_tf \|_{L^\infty(\mathbb{R}^n)}
 \le c_\beta t$, for all $t>0$, but $f\notin C^{0,1}_{\L_\mu}$.
\end{enumerate}
\end{thm}

It has no sense to take $\alpha=0$ as a H\"older exponent. By the
Campanato-type description of Proposition \ref{Prop:BMO y Calpha} we
see that the natural replacement in this situation is the space
$BMO_\L$.

\begin{thm}[Case $\alpha=0$]\label{Thm:alpha 0}
Assume that $q>n$.
\begin{enumerate}[(A)]
\item  A function $f$ is in $BMO_\L$ if and only if for $f$ being a function such that $f(x)(1+\abs{x})^{-(n+\varepsilon)}$ is integrable for any $\varepsilon>0$, and  for all $\beta>0$ there exists a
constant $c_\beta$ such that, for all balls $B$,
$$\left(\frac{1}{\abs{B}}\int_{\widehat{B}}|t^\beta\partial_t^\beta\P_tf(x)|^2~\frac{dx~dt}{t}\right)^{1/2}
\leq c_\beta.$$
\item Let $\L_\mu=-\Delta+\mu$, for $\mu>0$. There exists a function $f\in
BMO_{\L_\mu}$ such that, for some $\beta>0$, $\sup_{t>0}|t^\beta\partial_t^\beta\P_tf(0)|=\infty$.
\end{enumerate}
\end{thm}

We should notice that the proof of Theorem \ref{Thm:Operators} is relatively simple and it can be
presented rather quickly. This is in a big contrast with the proof given in \cite{Stinga-Torrea}
 for the case of the harmonic oscillator $H=-\Delta+|x|^2$. In \cite{Stinga-Torrea} pointwise
formulas of $H^{\pm\sigma}$ and Hermite-Riesz transforms must be
handled. In our proof of Theorem \ref{Thm:Operators}\textit{(a)} and
\textit{(b)} no Riesz transforms are needed. On the other hand, the
results in \cite{Stinga-Torrea} involve higher order spaces
$C^{k,\alpha}_H$. As we pointed out before, if we would like to have
higher order spaces then we should consider the spaces of the type
$\Lambda_\L^\alpha$ mentioned above. With such a description it is
very simple to extend the results of Theorem \ref{Thm:Operators} to
hold for all $\alpha,\sigma>0$ (with the appropriate relations
between them). But in this way still there is no pointwise
smoothness condition on the functions $f\in\Lambda^\alpha_\L$, which
are necessary in PDEs.

The organization of the paper is as follows. In Section \ref{Section:Boundedness}, in order to convince the reader how useful our method is, we present the proof of Theorem \ref{Thm:Operators}. In fact for those who are just interested
  in regularity properties of operators, this is the most important section.
  In Section \ref{Section:Preliminaries} we list a collection of estimates about Schr\"odinger
  kernels that we will need later. Some of them are known and we put them there  to make the
   paper more readable, but there are some new (although expectable) estimates, like those of
    Proposition \ref{Prop:Poisson est}. Section \ref{Section:BMO}  is a technical
    section about $BMO_\L^\alpha$ spaces and section \ref{Section:Proof} contains the proofs of Theorem \ref{Thm:Characterization}, \ref{Thm:alpha 1}  and \ref{Thm:alpha 0}.

Throughout this paper, the letters $c$ and $C$ denote positive constants that may change in each occurrence and they will depend on
the parameters involved (whenever it is necessary, we point out this
dependence with subscripts). The Gamma and Beta functions will be
denoted by $\Gamma$ and $\mathrm{B}$, respectively. Without
mentioning it, we will repeatedly apply the inequality $r^\eta
e^{-r}\leq C_\eta e^{-r/2}$, $\eta\geq0$, $r>0$. 

\section{Regularity of operators related to $\L$}\label{Section:Boundedness}

In this section we prove Theorem \ref{Thm:Operators}. First we need the following technical lemma.

\begin{lem}\label{crucial} Let $0< \gamma <1$, and $g$ be a continuous
 function such that $|g(x)| \le C\rho(x)^\gamma $, where $\rho $ is the critical radii function defined in \eqref{critical}. Then
\begin{enumerate}[(i)]
    \item For any $\varepsilon>0$, the function $g(x)(1+\abs{x})^{-(n+\gamma+\varepsilon)}$ is integrable.
    \item  For any $\beta > \gamma$ and any $N>0$ there exists a constant $C_{\beta,N,g}$ such that
        $$|s^\beta \partial^\beta_s\P_s g(x)|\le C_{\beta,N,g}\left(\rho(x)/s\right)^N\left(\rho(x)^{\gamma}+{s^{\gamma}}\right),\quad x\in\Real^n,~s>0.$$
    \item For any $N>0$ there exists a constant $C_{N,g}$ such that
        $$|\P_s g(x)|\le C_{N,g}\left(\rho(x)/s\right)^N\left(\rho(x)^{\gamma}+{s^{\gamma}}\right),\quad x\in\Real^n,~s>0.$$
\end{enumerate}
\end{lem}

\begin{proof}
Let us begin with $\textit{(i)}$. We have to check that the integrals
$$I=\int_{\abs{x}<2\rho(0)}\frac{\abs{g(x)}}{(1+\abs{x})^{n+\gamma+\varepsilon}}
~dx+
\sum_{j=1}^\infty\int_{2^j\rho(0)\leq\abs{x}<2^{j+1}\rho(0)}\frac{\abs{g(x)}}{(1+\abs{x})^{n+\gamma+\varepsilon}}~dx,$$
are finite. To that end we apply the hypothesis and some properties
of the function $\rho$ contained in Lemma \ref{Lem:equiv rho} below.
The  inequality $\abs{x}=\abs{x-0}<2^{j+1}\rho(0)$, $j\geq0$, and
the right inequality of \eqref{2} give us $\rho(x)\leq C\rho(0)2^j$.
Therefore,
$$I\leq C\rho(0)^{\gamma+n}+ C\sum_{j=1}^\infty\frac{\left(\rho(0)2^j\right)^{\gamma+n}}{\left(1+2^j\rho(0)\right)^{n+\gamma+\varepsilon}}\leq C+C\sum_{j=1}^\infty2^{-j\varepsilon}<\infty.$$

We {will only} prove \textit{(ii)}. The third statement
\textit{(iii)} can be proved in the same way. By \textit{(i)},
$\P_tg(x)$ is well defined. By Proposition \ref{Prop:Poisson
est}\textit{(b)} and Lemma \ref{Lem:equiv rho} below, for some
constant $C=C_{\beta,N,g}$, we have
\begin{align*}
    |s^\beta & \partial^\beta_s \P_sg(x)| \le C\int_{\Real^n}\frac{s^\beta\rho(x)^N}{(s+|x-y|)^{n+\beta+N}}~\rho(x)^\gamma
        \left(1+\frac{|x-y|}{\rho(x)}\right)^{\gamma}~dy \\
     &\le C\rho(x)^{\gamma+N}\int_{\Real^n}\frac{s^\beta}{(s+|x-y|)^{n+\beta+N}}~dy+C\rho(x)^N\int_{\Real^n}\frac{s^\beta}{(s+|x-y|)^{n+\beta+N-\gamma}}~dy \\
     &= C\rho(x)^{\gamma+N}s^{-N}+C \rho(x)^{N} s^{-N+\gamma}.
\end{align*}
\end{proof}

\begin{proof}[Proof of Theorem \ref{Thm:Operators}]
We start with the proof of part \textit{(a)}.  For $f\in
C^{0,\alpha}_\L$, we have
\begin{equation}\label{fi}
\L^{-\sigma/2}f(x)=\frac{1}{\Gamma(\sigma)}\int_0^\infty\P_sf(x)~\frac{ds}{s^{1-\sigma}},\quad x\in\Real^n.
\end{equation}
By Lemma \ref{crucial}\textit{(ii)}, since $\abs{f(x)}\leq C\rho(x)^\alpha$, we get 
\begin{align*}
\int_0^\infty |\P_sf(x)|\frac{ds}{s^{1-\sigma}}
&\le C\int_0^{\rho(x)}\left[\frac{\rho(x)^{\alpha+N_1}}{s^{N_1}}+\frac{\rho(x)^{N_1}}{s^{N_1-\alpha}}\right]\frac{ds}{s^{1-\sigma}} +C\int_{\rho(x)}^\infty\left[\frac{\rho(x)^{\alpha+N_2}}{s^{N_2}}+\frac{\rho(x)^{N_2}}{s^{N_2-\alpha}}\right]\frac{ds}{s^{1-\sigma}}\\
&\le C_{N_1,N_2,\alpha,f}\cdot\rho(x)^{\alpha+\sigma},
\end{align*}
by choosing $0<N_1<\sigma$ and $N_2>\alpha+\sigma$. Hence
$\L^{-\sigma/2}f(x)$ is well defined. Moreover, it satisfies the
required growth  $|\L^{-\sigma/2}f(x)|\leq
C\rho(x)^{\alpha+\sigma}$. So Lemma \ref{crucial} applies to it. Fix
any $\beta>\alpha+\sigma$. To obtain the conclusion we apply Theorem
\ref{Thm:Characterization}. That is, it is enough to prove that
$\|t^\beta\partial_t^\beta\P_t(\L^{-\sigma/2}f)\|_{L^\infty(\mathbb{R}^n)}\leq
C\norm{f}_{C^{0,\alpha}_\L}t^{\alpha+\sigma}$. By using formula
\eqref{fi} and Lemma \ref{crucial} together with Fubini's theorem,
we have
\begin{align*}
    t^\beta\partial_t^\beta\P_t(\L^{-\sigma/2}f)(x) = Ct^\beta\int_0^\infty\partial_t^\beta\P_t(\P_sf)(x)~\frac{ds}{s^{1-\sigma}}
     = Ct^\beta\int_0^\infty\partial^\beta_w\P_wf(x)\Big|_{w=t+s}~\frac{ds}{s^{1-\sigma}}.
\end{align*}
 Since $\beta>\alpha+\sigma$ we can use Theorem \ref{Thm:Characterization} to get \textit{(a)}:
\begin{align*}
    |t^\beta\partial_t^\beta\P_t(\L^{-\sigma/2}f)(x)|  & \leq C\norm{f}_{C^{0,\alpha}_\L}
    t^\beta\int_0^\infty(t+s)^{\alpha-\beta}\frac{ds}{s^{1-\sigma}}= C\norm{f}_{C^{0,\alpha}_\L}t^{\alpha+\sigma}\int_0^\infty(1+r)^{\alpha-\beta}~\frac{dr}{r^{1-\sigma}} \\
     &= C~\mathrm{B}(\sigma,\beta-\alpha-\sigma)\norm{f}_{C^{0,\alpha}_\L}t^{\alpha+\sigma},\quad\hbox{for all}~x\in\Real^n.
\end{align*}

To prove part \textit{(b)}, fix any $\beta>\alpha$. Since $0<\sigma<\alpha<1$ we can write
\begin{equation}\label{fl}
\L^{\sigma/2}f(x)=\frac{1}{\Gamma(-\sigma)}\int_0^\infty\left(\P_sf(x)-f(x)\right)~\frac{ds}{s^{1+\sigma}}=I(x,t)+II(x,t),
\end{equation}
where $I(x,t)$ is the part of the integral from $0$ to $t$. Since $f\in C^{0,\alpha}_\L$,
$$\abs{I(x,\rho(x))}\leq\int_0^{\rho(x)}\abs{\int_0^s\partial_r\P_rf(x)~dr}~\frac{ds}{s^{1+\sigma}}\leq C\int_0^{\rho(x)}\int_0^sr^{\alpha-1}~dr~\frac{ds}{s^{1+\sigma}}= C\rho(x)^{\alpha-\sigma}.$$
Taking $N=\alpha$ in Lemma \ref{crucial}\textit{(iii)} and using the growth of $f$ we also have
$$\abs{II(x,\rho(x))}\leq \int_{\rho(x)}^\infty\left(\abs{\P_sf(x)}+\abs{f(x)}\right)~\frac{ds}{s^{1+\sigma}}\leq C\int_{\rho(x)}^\infty\left[\frac{\rho(x)^{2\alpha}}{s^\alpha}+\rho(x)^\alpha\right]~\frac{ds}{s^{1+\sigma}}= C\rho(x)^{\alpha-\sigma}.$$
The computations above say that \eqref{fl} is well defined and that Theorem \ref{Thm:Characterization} can be applied to it. By linearity, it is enough to analyze
$t^\beta\partial_t^\beta\P_tI(x,t)$ and $t^\beta\partial_t^\beta\P_tII(x,t)$ separately. Note that
$$t^\beta\partial_t^\beta \P_tI(x,t)=\frac{t^\beta}{\Gamma(-\sigma)}\int_0^t\int_0^s\partial_w^{\beta+1}\P_wf(x)\big|_{w=t+r}~dr~\frac{ds}{s^{1+\sigma}}.$$
Apply Theorem \ref{Thm:Characterization} and the fact that $\beta>\alpha$ to obtain
\begin{align}
    \nonumber &|t^\beta\partial_t^\beta\P_tI(x,t)| \leq C\norm{f}_{C^{0,\alpha}_\L}t^\beta\int_0^t\int_0^s(t+r)^{\alpha-\beta-1}~dr~\frac{ds}{s^{1+\sigma}} \\
    \label{I} &= C\norm{f}_{C^{0,\alpha}_\L}t^\alpha\int_0^t\int_0^{s/t}(1+u)^{\alpha-\beta-1}~du~\frac{ds}{s^{1+\sigma}}\leq C\norm{f}_{C^{0,\alpha}_\L}t^\alpha\int_0^t\frac{s}{t}~\frac{ds}{s^{1+\sigma}}=C\norm{f}_{C^{0,\alpha}_\L}t^{\alpha-\sigma}.
\end{align}
Theorem \ref{Thm:Characterization} and Fubini's theorem give us
\begin{align}
    \nonumber|t^\beta\partial_t^\beta\P_tII(x,t)| &\leq C\int_t^\infty\left(
    \abs{t^\beta\partial_w^{\beta}\P_wf(x)\big|_{w=t+s}}+ \abs{t^\beta\partial_t^\beta\P_tf(x)}\right)~\frac{ds}{s^{1+\sigma}} \\
    \label{II} &\le C\norm{f}_{C^{0,\alpha}_\L}\int_t^\infty
    t^\beta (t+s)^{\alpha-\beta} + t^{\alpha} ~\frac{ds}{s^{1+\sigma}}=C\norm{f}_{C^{0,\alpha}_\L}t^{\alpha-\sigma}.
\end{align}
Collecting estimates \eqref{I} and \eqref{II} we get the conclusion of \textit{(b)}.

Let us finally check \textit{(c)}. Fix any $\beta>\alpha$. Note that we have $\displaystyle m(\L)f(x)=-\int_0^\infty\partial_s\P_sf(x)~a(s)~ds$. As $a$ is a bounded function and $f\in C^{0,\alpha}_\L$,
$$\int_0^{\rho(x)}\abs{\partial_s\P_sf(x)~a(s)}~ds\leq C\int_0^{\rho(x)}s^{\alpha-1}~ds=C\rho(x)^\alpha.$$
Moreover, by Lemma \ref{crucial}\textit{(ii)} with $\beta=1$ and
some $N>\alpha$ {at there}, we obtain 
$$\int_{\rho(x)}^\infty\abs{\partial_s\P_sf(x)~a(s)}~ds\leq C\int_{\rho(x)}^\infty
\left(\frac{\rho(x)}{s}\right)^N(\rho(x)^\alpha+s^\alpha)~\frac{ds}{s}=C\rho(x)^\alpha.$$
Therefore, $\abs{m(\L)f(x)}\leq C\rho(x)^\alpha$, so by Lemma
\ref{crucial}\textit{(i)} the hypothesis of Theorem
\ref{Thm:Characterization} holds for $m(\L)f$. By Theorem
\ref{Thm:Characterization} and Fubini's theorem we have
\begin{align*}
    |t^\beta\partial_t^\beta\P_t\big(m(\L)f\big)(x)| &= t^\beta\abs{\int_0^\infty
    \partial^{\beta+1}_w\P_wf(x)\big|_{w=t+s}~a(s)~ds}\leq C\norm{f}_{C^{0,\alpha}_\L}t^\beta\int_0^\infty(t+s)^{\alpha-(\beta+1)}~ds \\
     &= C\norm{f}_{C^{0,\alpha}_\L}t^\alpha\int_0^\infty(1+r)^{\alpha-(\beta+1)}~dr=C\norm{f}_{C^{0,\alpha}_\L}t^\alpha.
\end{align*}
\end{proof}

\section{Estimates on the kernels}\label{Section:Preliminaries}

The nonnegative potential $V$ in \eqref{L} satisfies a reverse
H\"older inequality for some $q>n/2$:
\begin{equation}\label{RHs}
\left(\frac{1}{\abs{B}}\int_BV(y)^q~dy\right)^{1/q}\leq\frac{C}{\abs{B}}\int_BV(y)~dy,
\end{equation}
for all balls $B\subset\Real^n$. Associated to this potential, Z.
Shen defines  the critical radii function in \cite{Shen} as
\begin{equation}\label{critical}
\rho(x):=\sup\Big\{r>0:\frac{1}{r^{n-2}}\int_{B(x,r)}V(y)~dy\leq1\Big\},\qquad x\in\Real^n.
\end{equation}

\begin{lem}[See {\cite[Lemma~1.4]{Shen}}]\label{Lem:equiv rho}
There exist $c>0$ and $k_0\geq1$ such that for all $x,y\in\Real^n$
\begin{equation}\label{2}
c^{-1}\rho(x)\left(1+\frac{\abs{x-y}}{\rho(x)}\right)^{-k_0}\leq\rho(y)\leq c\rho(x)
\left(1+\frac{\abs{x-y}}{\rho(x)}\right)^{\frac{k_0}{k_0+1}}.
\end{equation}
\end{lem}

Let $\set{\T_t}_{t>0}$ be the heat--diffusion semigroup associated to $\L$:
\begin{equation}\label{heatL}
\T_tf(x)\equiv e^{-t\L}f(x)=\int_{\Real^n}k_t(x,y)f(y)~dy,\qquad f\in L^2(\Real^n),~x\in\Real^n,~t>0.
\end{equation}

\begin{lem}[See \cite{Dziubanski-Zienkiewicz Hp,Kurata}]\label{Lem:cota heat L}
For every $N>0$ there exists a constant $C_N$ such that
\begin{equation}\label{1}
0\leq k_t(x,y)\leq C_Nt^{-n/2}e^{-\frac{\abs{x-y}^2}{5t}}\left(1+\frac{\sqrt{t}}{\rho(x)}
+\frac{\sqrt{t}}{\rho(y)}\right)^{-N},\quad x,y\in\Real^n,~t>0.
\end{equation}
\end{lem}

The kernel of the classical heat semigroup $\set{T_t}_{t>0}=\{e^{t\Delta}\}_{t>0}$ on $\Real^n$ is
\begin{equation}\label{classical heat}
h_t(x):=\frac{1}{(4\pi t)^{n/2}}~e^{-\frac{\abs{x}^2}{4t}},\qquad x\in\Real^n,~t>0,
\end{equation}

\begin{lem}[See {\cite[Proposition~2.16]{Dziubanski-Zienkiewicz Hp}}]\label{Lem:Schwartz}
There exists a nonnegative function $\omega\in\mathcal{S}$, where  $\mathcal{S}$ denotes the Schwartz's class of rapidly decreasing $C^\infty$ functions in $\Real^n$,  such that
\begin{equation}\label{3}
\abs{k_t(x,y)-h_t(x-y)}\leq\left(\frac{\sqrt{t}}{\rho(x)}\right)^\delta\omega_t(x-y),\quad x,y\in\Real^n,~t>0,
\end{equation}
where $\omega_t(x-y):=t^{-n/2}\omega\left((x-y)/\sqrt{t}\right)$ and
$\delta:=2-\frac{n}{q}>0$.
\end{lem}

We define the following kernel that will be useful in the sequel. Let
\begin{equation}\label{Qt}
Q_t(x,y):=t^2\left.\frac{\partial k_s(x,y)}{\partial s}\right|_{s=t^2},\quad x,y\in\Real^n,~t>0.
\end{equation}

\begin{lem}[See {\cite[Proposition~4]{DGMTZ}}]\label{Lem:Q est}
Let $\delta$ be as in Lemma \ref{Lem:Schwartz}. There exists a constant $c$ such that for every $N$ there is a constant $C_N$ such that
\begin{enumerate}[(a)]
    \item $\displaystyle \abs{Q_t(x,y)}\leq C_Nt^{-n}e^{-c\frac{\abs{x-y}^2}{t^2}}\left(1+\frac{t}{\rho(x)}+\frac{t}{\rho(y)}\right)^{-N}$;
    \item $\displaystyle \abs{Q_t(x+h,y)-Q_t(x,y)}\leq C_N\left(\frac{\abs{h}}{t}\right)^\delta
    t^{-n}e^{-c\frac{\abs{x-y}^2}{t^2}}\left(1+\frac{t}{\rho(x)}+\frac{t}{\rho(y)}\right)^{-N}$, for all $\abs{h}\leq t$;
    \item $\displaystyle\abs{\int_{\Real^n}Q_t(x,y)~dy}\leq C_N\frac{(t/\rho(x))^\delta}{(1+t/\rho(x))^N}$.
\end{enumerate}
\end{lem}

\begin{rem}\label{Rem:delta prime}
Let $0<\delta'\leq\delta$. Then we can easily deduce from Lemma \ref{Lem:Q
est}\textit{(c)} that for any $N>0$ there exists a constant $C_N$
such that $\displaystyle\abs{\int_{\Real^n}Q_t(x,y)~dy}\leq C_N\frac{(t/\rho(x))^{\delta'}}{(1+t/\rho(x))^N}$.
\end{rem}

Using the heat semigroup \eqref{heatL} and through Bochner's subordination formula, see \cite{SteinTopics}, we have:
\begin{equation}\label{subordinacion}
\P_tf(x)\equiv
e^{-t\sqrt{\L}}f(x)=\frac{1}{\sqrt{\pi}}\int_0^\infty\frac{e^{-u}}{\sqrt{u}}~\T_{t^2/(4u)}f(x)~du=\frac{t}{2\sqrt\pi}\int_0^\infty
\frac{~e^{-{t^2}/{(4u)}}}{u^{3/2}}\mathcal{T}_{u}f(x)~du,
\end{equation}
for any $ x\in\Real^n,~t>0.$ It follows that the $\L$--Poisson kernel is given by
\begin{equation}\label{Poisson kernel}
\P_t(x,y)=\frac{1}{\sqrt{\pi}}\int_0^\infty\frac{e^{-u}}{\sqrt{u}}~k_{t^2/(4u)}(x,y)~du=\frac{t}{2\sqrt\pi}\int_0^\infty
\frac{~e^{-{t^2}/{(4u)}}}{u^{3/2}}~k_u(x,y)~du.
\end{equation}
We will denote {the classical Poisson semigroup in $\Real^{n+1}_+$
by $P_tf(x)=P_t\ast f(x)$}, , where
\begin{equation}\label{classical Poisson kernel}
P_t(x)=c_n\frac{t}{(t^2+\abs{x}^2)^{\frac{n+1}{2}}}.
\end{equation}

Let us now compute the fractional derivatives \eqref{frac deriv} of
the Poisson kernel. The formula will involve the kernel $Q_t(x,y)$
of \eqref{Qt} and the Hermite polynomials $H_m(r)$ defined, for
$m\in\mathbb{N}_0$ and $r\in\Real$, as $H_m(r)=(-1)^me^{r^2}\tfrac{d^m}{dr^m}(e^{-r^2})$. From the first identity in \eqref{Poisson kernel} and the definition of $Q_t$ in \eqref{Qt}, we have
$$\partial_t\P_t(x,y)=\frac{2}{t\sqrt{\pi}}\int_0^\infty\frac{e^{-u}}{\sqrt{u}}~Q_{t/(2\sqrt{u})}(x,y)~du
= \frac{2}{\sqrt{\pi}}\int_0^\infty
e^{-{t^2}/{(4v^2)}}Q_v(x,y)~\frac{dv}{v^2}.$$
Hence, for each $m\geq1$, we obtain
$$\partial_t^m\P_t(x,y)=\frac{2(-1)^m}{\sqrt{\pi}}\int_0^\infty H_{m-1}\left(\frac{t}{2v}\right)e^{-\frac{t^2}{4v^2}}\frac{1}{(2v)^{m-1}}~Q_v(x,y)~\frac{dv}{v^2}.$$
With this we can write the derivatives $\partial_t^\beta\P_t(x,y)$, $\beta>0$, as follows. For $m=[\beta]+1$,
\begin{align}
    \nonumber\partial_t^\beta\P_t(x,y) &= \frac{e^{-i \pi(m-\beta)}}{\Gamma(m-\beta)}\int_0^\infty\partial_t^m\P_{t+s}(x,y)s^{m-\beta}~\frac{ds}{s} \\
    \label{deriv Poisson}&= \frac{2(-1)^{m}e^{-i \pi(m-\beta)}}{\Gamma(m-\beta)\sqrt{\pi}}\int_0^\infty\int_0^\infty
     H_{m-1}\left(\frac{t+s}{2v}\right)e^{-\frac{(t+s)^2}{4v^2}}\frac{1}{(2v)^{m-1}}~Q_v(x,y)~\frac{dv}{v^2}s^{m-\beta}~\frac{ds}{s} \\
      \nonumber&= \frac{2(-1)^{m}e^{-i \pi(m-\beta)}}{\Gamma(m-\beta)\sqrt{\pi}}\int_0^\infty
     \left[\int_0^\infty H_{m-1}\left(\frac{t+s}{2v}\right)e^{-\frac{(t+s)^2}{4v^2}}s^{m-\beta}
     ~\frac{ds}{s}\right]\frac{1}{(2v)^{m-1}}~Q_v(x,y)~\frac{dv}{v^2}.
\end{align}

\begin{prop}\label{Prop:Poisson est}
Let $\beta>0$. For any $0<\delta'\leq\delta$ with $0<\delta'<\beta$,
and $N>0$ there exists a constant $C=C_{N,\beta,\delta'}$ such that
\begin{enumerate}[(a)]
    \item $\displaystyle\abs{\P_t(x,y)}
            \leq C \frac{t}{(\abs{x-y}^2+t^2)^{\frac{n+1}{2}}}
             \left(1+\frac{(\abs{x-y}^2+t^2)^{1/2}}{\rho(x)}
             + \frac{(\abs{x-y}^2+t^2)^{1/2}}{\rho(y)}\right)^{-N}$;
    \item $\displaystyle|t^\beta\partial_t^\beta\P_t(x,y)|\leq C \frac{t^\beta}{(\abs{x-y}^2+t^2)^{\frac{n+\beta}{2}}}
          \left(1+\frac{(\abs{x-y}^2+t^2)^{1/2}}{\rho(x)}
          + \frac{(\abs{x-y}^2+t^2)^{1/2}}{\rho(y)}\right)^{-N}$;
    \item  For all $\abs{h}\leq t$,
        \begin{multline*}
            \quad|t^\beta\partial_t^\beta\P_t(x+h,y)-t^\beta\partial_t^\beta\P(x,y)| \\
             \leq C\left(\frac{\abs{h}}{t}\right)^{\delta'}\frac{t^\beta}{(\abs{x-y}^2+t^2)^{\frac{n+\beta}{2}}}
              \left(1+\frac{(\abs{x-y}^2+t^2)^{1/2}}{\rho(x)}+ \frac{(\abs{x-y}^2+t^2)^{1/2}}{\rho(y)}\right)^{-N};
        \end{multline*}
    \item $\displaystyle\abs{\int_{\Real^n}t^\beta\partial_t^\beta\P_t(x,y)~dy}
           \leq C\frac{(t/\rho(x))^{\delta'}}{(1+t/\rho(x))^N}.$
\end{enumerate}
\end{prop}

\begin{proof}
Let us prove \textit{(a)} first. Observe that, by the second
identity of \eqref{Poisson kernel} and Lemma \ref{Lem:cota heat L}, we obtain
\begin{align*}
    \abs{\P_t(x,y)} &\leq Ct\int_0^{\abs{x-y}^2+t^2} {u^{-\frac{n+3}{2}}}{~e^{-\frac{\abs{x-y}^2+t^2}{cu}}}
     \left(1+\frac{\sqrt{u}}{\rho(x)}+\frac{\sqrt{u}}{\rho(y)}\right)^{-N}du \\
     &\quad +Ct\int_{\abs{x-y}^2+t^2}^\infty {u^{-\frac{n+3}{2}}}{~e^{-\frac{\abs{x-y}^2+t^2}{cu}}}
     \left(1+\frac{\sqrt{u}}{\rho(x)}+\frac{\sqrt{u}}{\rho(y)}\right)^{-N}du=: I+II.
\end{align*}
For $I$ apply the change of variables $r=(\abs{x-y}^2+t^2)/u$ to get
$$I\leq\frac{Ct}{(\abs{x-y}^2+t^2)^{\frac{n+1}{2}}}\left(1+\frac{(\abs{x-y}^2+t^2)^{1/2}}{\rho(x)}
+\frac{(\abs{x-y}^2+t^2)^{1/2}}{\rho(y)}\right)^{-N} \int_1^\infty
r^{\frac{n+N-1}{2}}e^{-cr}~{dr}.$$ For $II$,
$$II\leq Ct\left(1+\frac{(\abs{x-y}^2+t^2)^{1/2}}{\rho(x)}+\frac{(\abs{x-y}^2+t^2)^{1/2}}{\rho(y)}\right)^{-N}\int_{\abs{x-y}^2+t^2}^\infty {u^{-\frac{n+3}{2}}}~du.$$
Combining these last two estimates we conclude the proof of
\textit{(a)}.

To prove \textit{(b)}, note that we can estimate the integral in
brackets in \eqref{deriv Poisson} as follows:
\begin{align}
    \nonumber \abs{\int_0^\infty
    H_{m-1}\left(\frac{t+s}{2v}\right)e^{-\frac{(t+s)^2}{4v^2}}s^{m-\beta}~\frac{ds}{s}}
     &\leq C_m\int_0^\infty e^{-c\frac{(t+s)^2}{4v^2}}s^{m-\beta}~\frac{ds}{s}
     \leq C_me^{-c\frac{t^2}{v^2}}\int_0^\infty
     e^{-c\frac{s^2}{v^2}}s^{m-\beta}~\frac{ds}{s}\\
     \label{est for beta deriv}&= C_me^{-c\frac{t^2}{v^2}}v^{m-\beta}\int_0^\infty e^{-cr^2}r^{m-\beta}~\frac{dr}{r}
     = C_{m,\beta}~e^{-c\frac{t^2}{v^2}}v^{m-\beta}.
\end{align}
Using identity \eqref{deriv Poisson}, this last inequality and Lemma \ref{Lem:Q est}\textit{(a)}, we get
$$|\partial_t^\beta\P_t(x,y)| \leq C\int_0^\infty e^{-c\frac{t^2}{v^2}} v^{-\beta}\abs{Q_v(x,y)}~\frac{dv}{v}\leq C\int_0^\infty \frac{e^{-c\frac{\abs{x-y}^2+t^2}{v^2}}}{v^{n+\beta}}\left(1+\frac{v}{\rho(x)}+\frac{v}{\rho(y)}\right)^{-N}\frac{dv}{v}.$$
The last integral can be split and treated as $I$ and $II$ above. Hence \textit{(b)} is proved.

The proof of part \textit{(c)} follows parallel lines as we have just done for \textit{(b)} by using identity
\eqref{deriv Poisson}, estimate \eqref{est for beta deriv} and Lemma \ref{Lem:Q est}\textit{(b)}.

For \textit{(d)}, let $0<\delta'\leq\delta$ with $0<\delta'<\beta$.
By Remark \ref{Rem:delta prime} and the change of variables $w=t/v$,
\begin{multline*}
    \abs{\int_{\Real^n}t^\beta\partial_t^\beta\P_t(x,y)~dy} \leq Ct^\beta\int_0^\infty e^{-c\frac{t^2}{v^2}}v^{-\beta}
    \abs{\int_{\Real^n}Q_v(x,y)~dy}~\frac{dv}{v} \\
     \leq Ct^\beta\int_0^\infty e^{-c\frac{t^2}{v^2}}v^{-\beta}\frac{(v/\rho(x))^{\delta'}}{(1+v/\rho(x))^N}~\frac{dv}{v}= C(t/\rho(x))^{\delta'}\int_0^\infty e^{-cw^2}\frac{w^{\beta-\delta'}}{(1+t/(w\rho(x)))^N}~\frac{dw}{w}.
\end{multline*}
On one hand,
\begin{align*}
    \int_{t/\rho(x)}^\infty e^{-cw^2}\frac{w^{\beta-\delta'}}{(1+t/(w\rho(x)))^N}~\frac{dw}{w}
     &\leq e^{-c\frac{t^2}{2\rho(x)^2}}\int_0^\infty e^{-c\frac{w^2}{2}}w^{\beta-\delta'}~\frac{dw}{w}
     &\leq Ce^{-c\frac{t^2}{\rho(x)^2}}\leq \frac{C}{(1+t/\rho(x))^N}.
\end{align*}
On the other hand, we consider two cases. If $t/\rho(x)\leq1$ then
$$\int_0^{t/\rho(x)} e^{-cw^2}\frac{w^{\beta-\delta'}}{(1+t/(w\rho(x)))^N}~\frac{dw}{w}\leq \int_0^1w^{
\beta-\delta'}~\frac{dw}{w}\leq\frac{C}{(1+t/\rho(x))^N}.$$
If $t/\rho(x)>1$ then
\begin{align*}
    \int_0^{t/\rho(x)} e^{-cw^2}\frac{w^{\beta-\delta'}}{(1+t/(w\rho(x)))^N}~\frac{dw}{w}
    \leq \frac{1}{(t/\rho(x))^N}\int_0^\infty e^{-cw^2}w^{\beta-\delta'+N}~\frac{dw}{w}
    \leq \frac{C}{(1+t/\rho(x))^N}.
\end{align*}
This concludes the proof of the proposition.
\end{proof}

To finish this section we show a reproducing formula for the
operator $t^\beta\partial_t^\beta\P_t$ on $L^2(\Real^n)$.

\begin{lem}\label{Lem:equality L2}
The operator $t^\beta\partial_t^\beta\P_t$ defines an isometry from
$L^2(\Real^n)$ into $L^2(\Real^{n+1}_+,\frac{dx~dt}{t})$. Moreover,
\begin{equation}\label{equ4}
f(x)=\frac{4^\beta}{\Gamma(2\beta)}\lim_{\substack{\varepsilon\to
0\\N\to\infty}}\int_\varepsilon^N(t^\beta\partial_t^\beta\P_t)^2f(x)~\frac{dt}{t},\quad\hbox{in}~L^2(\Real^n).
\end{equation}
\end{lem}

\begin{proof}
The proof is standard by using spectral techniques, see for instance
\cite{DGMTZ}, and we omit it here.
\end{proof}

\section{The Campanato-type space $BMO_\L^\alpha$, $0\leq\alpha\leq1$: duality and pointwise description}\label{Section:BMO}

In this section we give the definition of space $BMO^\alpha_\L$ introduced in \cite{Bongioanni-Harboure-Salinas-weighted}, the relation with $C^{0,\alpha}_\L$ and the duality result $H^p_\L$--$BMO^\alpha_\L$.

\begin{defn}[$BMO^\alpha$ space for $\L$, see \cite{Bongioanni-Harboure-Salinas-weighted}]
A locally integrable function $f$ is in $BMO_\L^\alpha$,
$0\leq\alpha\leq1$, if there exists a constant $C$ such that
\begin{enumerate}
\item[(i)] $\displaystyle\frac{1}{\abs{B}}\int_B\abs{f(x)-f_B}~dx\leq C\abs{B}^{\frac{\alpha}{n}}$, for every ball $B$ in $\Real^n$, and
\item[(ii)] $\displaystyle\frac{1}{\abs{B}}\int_B\abs{f(x)}~dx\leq C\abs{B}^{\frac{\alpha}{n}}$,
 for every $B=B(x_0,r_0)$, where $x_0\in\Real^n$ and $r_0\geq\rho(x_0)$.
\end{enumerate}
As usual, $f_{B}:=\displaystyle\frac{1}{|B|}\int_Bf(x)~dx$. The norm $\norm{f}_{BMO_\L^\alpha}$ is defined as the infimum of the constants $C$ such that (i) and (ii) above hold.
\end{defn}

\begin{rem}\label{Rem:equiv norm p}
The space $BMO_\L^0$ is the $BMO$ space naturally associated to $\L$
given in \cite{DGMTZ}. We require $\alpha\leq1$ in the definition
above because if $\alpha>1$ then the space only contains constant
functions. By using the classical John-Nirenberg inequality it can be seen that
if in (i) and (ii) $L^1$-norms are replaced by $L^p$-norms, for
$1<p<\infty$, then the space $BMO_\L^\alpha$ does not change.
\end{rem}

\begin{prop}\label{Prop:from small balls}
Let $f\in BMO_\L^\alpha$, $0<\alpha\leq1$, and $B=B(x,r)$ with
$r<\rho(x)$. Then there exists a constant $C=C_{\alpha}$ such that $\abs{f_B}\leq C_{\alpha}\norm{f}_{BMO_\L^\alpha}\rho(x)^\alpha$.
\end{prop}

\begin{proof}
Let $j_0$ be a positive integer such that $2^{j_0}r\leq\rho(x)<2^{j_0+1}r$. Since $f\in BMO^\alpha_\L$, we have
\begin{align*}
    \abs{f_B} &\leq \frac{1}{\abs{B}}\int_B\abs{f(z)-f_{2B}}~dz+\sum_{j=1}^{j_0}\abs{f_{2^jB}-f_{2^{j+1}B}}+\abs{f_{2^{j_0+1}B}} \\
     &\leq C\norm{f}_{BMO_\L^\alpha}\abs{B}^{\frac{\alpha}{n}}\sum_{j=1}^{j_0+1}
     \left(2^\alpha\right)^j= C\norm{f}_{BMO_\L^\alpha}\abs{B}^{\frac{\alpha}{n}}\frac{2^\alpha-2^{\alpha(j_0+1)}}{1-2^\alpha} \\
     &\leq C\norm{f}_{BMO_\L^\alpha}\abs{B}^{\frac{\alpha}{n}}2^{\alpha(j_0+1)}=
     C2^\alpha\norm{f}_{BMO_\L^\alpha}\left(2^{j_0}r\right)^\alpha\leq C_\alpha\norm{f}_{BMO^\alpha_\L}\rho(x)^\alpha.
\end{align*}
\end{proof}

\begin{rem}\label{rem:3.4}
From the proof of Proposition \ref{Prop:from small balls} it can be
seen that if $f$ is in $BMO_\L=BMO_\L^0$ and $B=B(x,r)$ with
$r<\rho(x)$ then the conclusion of Lemma 2 in \cite{DGMTZ} follows:
$$\abs{f_B}\leq C\left(1+\log\frac{\rho(x)}{r}\right)\norm{f}_{BMO_\L}.$$
\end{rem}

Following the works by J. Dziuba\'nski and J. Zienkiewicz
\cite{Dziubanski-Zienkiewicz,Dziubanski-Zienkiewicz
Colloquium,Dziubanski-Zienkiewicz Hp} we introduce the Hardy space
naturally associated to $\L$. An integrable function $f$ is an
element of the $\L$--Hardy space $H^p_\L$, $0<p\leq1$, if the
maximal function $\T^\ast f(x):=\sup_{s>0}\abs{\T_sf(x)}$, see
\eqref{heatL}, belongs to $L^p(\Real^n)$. The quasi-norm in $H^p_\L$
is defined by $\norm{f}_{H^p_\L}:=\norm{\T^\ast f}_{L^p(\Real^n)}$.
In \cite{Dziubanski-Zienkiewicz,Dziubanski-Zienkiewicz Hp} the
atomic description of $H^p_\L$ was given. Let
$\widetilde{\delta}=\min\set{1,\delta}$, with $\delta$ as in Lemma
\ref{Lem:Schwartz}. An atom of the $\L$--Hardy space $H^p_\L$,
$\tfrac{n}{n+\widetilde{\delta}}<p\leq1$, associated with a ball
$B(x_0,r)$ is a function $a$ such that $\supp a\subseteq B(x_0,r)$
with $r\leq\rho(x_0)$,
$\norm{a}_{L^\infty}\leq\abs{B(x_0,r)}^{-1/p}$ and, if
$r<\rho(x_0)/4$ then $\displaystyle\int a(x)~dx=0$. The atomic
$\L$--Hardy space $H^p_{\mathrm{at},\L}$,
$\tfrac{n}{n+\widetilde{\delta}}<p\leq1$, is defined as the set of
$L^1$-functions $f$ with compact support such that $f$ can be
written as a sum $f=\sum_i\lambda_ia_i$, where $\lambda_i$ are
complex numbers with $\sum_i\abs{\lambda_i}<\infty$ and $a_i$ are
atoms in $H^p_\L$. The quasi-norm in the atomic Hardy space, namely
the infimum of all such possible $\sum_i\abs{\lambda_i}$, turns out
to be equivalent to the quasi-norm $\norm{f}_{H^p_\L}$, for that
range of $p$. When $n/2<q<n$, such equivalence can be extended to
hold for Hardy spaces $H^p_\L$ with
$\tfrac{n}{n+1}<p\leq\tfrac{n}{n+\delta}$, but atoms must be
redefined, see \cite{Dziubanski-Zienkiewicz Colloquium}.

As mentioned in \cite{Bongioanni-Harboure-Salinas-weighted}, see
also \cite{Harboure-Salinas-Viviani} and \cite{YYZ}, once an atomic
decomposition of $H^p_\L$ is at hand, the dual space can be easily
described. We present the following result without proof.

\begin{thm}[Duality $H^p_\L$--$BMO_\L^\alpha$]\label{Thm:Duality}
Let $q>n$ and $0\leq\alpha<1$. Then the dual of
$H^{\frac{n}{n+\alpha}}_\L$ is the space $BMO_\L^\alpha$.
 More precisely, any continuous linear functional $\ell$ over $H^{\frac{n}{n+\alpha}}_\L$ can be represented as
$$\ell(a)=\int_{\Real^n}f(x)a(x)~dx,$$
for some function $f\in BMO_\L^\alpha$ and all atoms $a\in H^{\frac{n}{n+\alpha}}_\L$.
Moreover, $\norm{\ell}\sim\norm{f}_{BMO_\L^\alpha}$.
\end{thm}

\begin{prop}[Campanato-type description of $C^{0,\alpha}_\L$]\label{Prop:BMO y Calpha}
If $0<\alpha\leq1$ then the spaces $BMO_\L^\alpha$ and
$C^{0,\alpha}_\L$ are equal and their norms are equivalent.
\end{prop}

The previous result was proved in \cite[Proposition~4]{Bongioanni-Harboure-Salinas-weighted} for $0<\alpha<1$ and in a weighted context. We just mention here that the proof given there is also valid for $\alpha=1$. As a consequence, the functions in $BMO^\alpha_\L$ can be modified in a set of measure zero so they become $\alpha$-H\"older continuous, $0<\alpha\leq1$.

\section{Proofs of Theorems \ref{Thm:Characterization}, \ref{Thm:alpha 1} and \ref{Thm:alpha 0}}\label{Section:Proof}

The proof of Theorem \ref{Thm:Characterization} will follow the scheme \textit{(i)} $\Longrightarrow$ \textit{(ii)} $\Longrightarrow$ \textit{(iii)} $\Longrightarrow$ \textit{(i)}. The statement \textit{(iii)} $\Longrightarrow$ \textit{(i)} relies heavily on the duality $H^{\frac{n}{n+\alpha}}_\L-BMO_\L^\alpha$ developed in Section \ref{Section:BMO}, so the method, rather technical, will work only for $0<\alpha<1$. Observe that the proof of \textit{(ii)} $\Longrightarrow$ \textit{(iii)} is immediate. To prove Theorem \ref{Thm:alpha 1}\textit{(I)} we just note that the proofs of \textit{(i)} $\Longrightarrow$ \textit{(ii)} $\Longrightarrow$ \textit{(iii)} in Theorem \ref{Thm:Characterization} also hold for $\alpha=1$. A simple contradiction argument shows that the converse is false: if it were true then, by the comment just made, $f\in C^{0,1}_\L$ would be equivalent to \textit{(ii)} in Theorem \ref{Thm:Characterization} with $\alpha=1$. But that contradicts the statement of Theorem \ref{Thm:alpha 1}\textit{(II)} (which is proved by a counterexample). For Theorem \ref{Thm:alpha 0}\textit{(A)} we only have to prove the necessity part since the sufficiency for $\beta=1$ follows the same lines as in \cite{DGMTZ}. For part \textit{(B)} we give a counterexample.

\subsection{Proof of Theorem \ref{Thm:Characterization}: \textit{(i)}$\Longrightarrow$\textit{(ii)}}

Let $f\in C^{0,\alpha}_\L$. Then
\begin{multline*}
    |t^\beta  \partial_t^\beta\P_tf(x)| = \abs{\int_{\Real^n}t^\beta\partial_t^\beta
    \P_t(x,z)\left(f(z)-f(x)\right)~dz+f(x)\int_{\Real^n}t^\beta\partial_t^\beta\P_t(x,z)~dz} \\
     \leq \norm{f}_{C^{0,\alpha}_\L}\int_{\Real^n}|t^\beta\partial_t^\beta\P_t(x,z)|\abs{x-z}^\alpha
     dz+\norm{f}_{C^{0,\alpha}_\L}\rho(x)^\alpha\abs{\int_{\Real^n}t^\beta\partial_t^\beta\P_t(x,z)~dz}=: I+II.
\end{multline*}
Applying Proposition \ref{Prop:Poisson est}\textit{(b)}, {we
obtain}
$$I\leq C\norm{f}_{C^{0,\alpha}_\L}\int_{\Real^n}\frac{t^\beta\abs{x-z}^\alpha}{\left(t+\abs{x-z}\right)^{n+\beta}}~dz=
C\norm{f}_{C^{0,\alpha}_\L}t^\alpha.$$ For $II$ we consider two
cases. Assume first that $\rho(x)\leq t$. Then {Proposition
\ref{Prop:Poisson est}\textit{(b)}} gives
$$II\leq C{\norm{f}_{C^{0,\alpha}_\L}}t^\alpha\int_{\Real^n}\frac{t^\beta}{(t+\abs{x-z})^{n+\beta}}~dz= C{\norm{f}_{C^{0,\alpha}_\L}}{t^\alpha}.$$
Suppose now that $\rho(x)>t$. Since $s>n$, we have $\delta>1$ in Lemma \ref{Lem:Schwartz}. Therefore we can choose $\delta'$ such that
$\alpha<\delta'\leq\delta$ with $\delta'<\beta$. By Proposition
\ref{Prop:Poisson est}\textit{(d)}, $II\leq C\norm{f}_{C^{0,\alpha}_\L}t^\alpha(t/\rho(x))^{\delta'-\alpha}\leq C\norm{f}_{C^{0,\alpha}_\L}{t^{\alpha}}$.

\subsection{Proof of Theorem \ref{Thm:Characterization}: \textit{(iii)}$\Longrightarrow$\textit{(i)}}

Assume that $f\in L^1(\Real^n,(1+\abs{x})^{-(n+\alpha+\varepsilon)}~dx)$ for any
$0<\varepsilon<\min\{\beta-\alpha,1-\alpha\}$, and that the Carleson condition in \textit{(iii)}
holds. Let
$$[d\mu_f]_{\alpha,\beta}:=\sup_B\frac{1}{\abs{B}^{\frac{\alpha}{n}}}\left(\frac{1}{\abs{B}}
 \int_{\widehat{B}}|t^\beta\partial_t^\beta\P_tf(x)|^2~\frac{dx~dt}{t}\right)^{1/2}.$$
To show that $f\in BMO^\alpha_\L$, by Theorem \ref{Thm:Duality}, it is enough to prove that the linear
 functional that maps each $g\in H^{\frac{n}{n+\alpha}}_\L$ to $\displaystyle\Phi_f(g):=\int_{\Real^n}f(x)g(x)~dx$ is
 continuous on $H^{\frac{n}{n+\alpha}}_\L$. In fact, we are going to prove that
  $\abs{\Phi_f(g)}\leq C[d\mu_f]_{\alpha,\beta}\norm{g}_{H^{\frac{n}{n+\alpha}}_\L}$,
   which implies that $f\in BMO_\L^\alpha$ with $\norm{f}_{BMO^\alpha_\L}\leq C[d\mu_f]_{\alpha,\beta}$.
\newline\noindent\textbf{Step 1.} It consists in writing the functional
$\Phi$ by using extensions of $f$ and $g$ to the upper half-space.
Define, for $x\in\Real^n$, $t>0$, the extended functions $F(x,t):=t^\beta\partial_t^\beta\P_tf(x)$ and $G(x,t):=t^\beta\partial_t^\beta\P_tg(x)$.

\begin{lem}\label{Lem:reproduce}
Let $f\in L^1(\Real^n,(1+\abs{x})^{-(n+\alpha+\varepsilon)}dx)$ for any $\varepsilon >0$ and $g$ be an
$H^{\frac{n}{n+\alpha}}_\L$--atom. Then
$$\frac{4^{\beta}}{{\Gamma(2\beta)}}\int_{\Real^n}f(x)\overline{g(x)}~dx=\int_{\Real^{n+1}_+}F(x,t)\overline{G(x,t)}~\frac{dx~dt}{t}.$$
\end{lem}

The rather technical proof of the lemma above will be given at the end of this subsection. To continue we assume its validity. So we are reduced to study the integral in the right-hand side.
\newline\noindent\textbf{Step 2.} To handle the integral in Lemma \ref{Lem:reproduce} we take a result
of E. Harboure, O. Salinas and B. Viviani about tent spaces into our particular case.

\begin{lem}[See {\cite[p.~279]{Harboure-Salinas-Viviani}}]\label{Lem:Pola}
For any pair of measurable functions $F$ and $G$ on $\Real^{n+1}_+$
we have
\begin{multline*}
\int_{\Real^{n+1}_+}\abs{F(x,t)}\abs{G(x,t)}~\frac{dx~dt}{t} \\
\leq
C\sup_B\left(\frac{1}{\abs{B}^{1+\frac{2\alpha}{n}}}\int_{\widehat{B}}\abs{F(x,t)}^2~\frac{dx~dt}{t}\right)^{1/2}\times
\left(\int_{\Real^n}\left(\int_{\mathbf{\Gamma}(x)}\abs{G(y,t)}^2~\frac{dy~dt}{t^{n+1}}\right)^{\frac{n}{2(n+\alpha)}}dx\right)^{\frac{n+\alpha}{n}},
\end{multline*}
where $\mathbf{\Gamma}(x)$ denotes the cone with vertex at $x$ and aperture $1$: $\left\{(y,t)\in\Real_+^{n+1}: |x-y|<t \right\}$.
\end{lem}

If we take $F(x,t)=t^\beta\partial_t^\beta\P_tf(x)$ in Lemma \ref{Lem:Pola} then the supremum that appears in the inequality is exactly $[d\mu_f]_{\alpha,\beta}$. Hence it remains to handle the term with $G(x,t)$, which is done in the last step.
\newline\noindent\textbf{Step 3.} The area function $S_\beta$ defined by
\begin{equation}\label{area function}
S_{\beta}(h)(z)=\left(\iint_{\mathbf{\Gamma}(z)}|t^\beta\partial_t^\beta\P_th(y)|^2\frac{dy~dt}{t^{n+1}}\right)^{1/2},\quad
z\in\Real^n,
\end{equation}
is a bounded operator on $L^2(\Real^n)$. Indeed, by the Spectral
Theorem, the square function
\begin{equation}\label{square function}
g_\beta(h)(x)=\left(\int_0^\infty|t^\beta\partial_t^\beta\P_th(x)|^2~\frac{dt}{t}\right)^{1/2},\quad x\in\Real^n,
\end{equation}
satisfies $\norm{g_\beta(h)}_{L^2(\Real^n)}=\Gamma(\beta)\norm{h}_{L^2(\Real^n)}$ and it is
easy to check that $\norm{S_\beta(h)}_{L^2(\Real^n)}=\norm{g_\beta(h)}_{L^2(\Real^n)}$. We will finish the proof of \textit{(iii)} $\Longrightarrow$ \textit{(i)} in Theorem \ref{Thm:Characterization} as soon as we have proved the following

\begin{lem}\label{lem:est S}
There exists a constant $C$ such that for any function $g$ which is
a linear combination of $H^{\frac{n}{n+\alpha}}_\L$--atoms we have $\norm{S_\beta(g)}_{L^{\frac{n}{n+\alpha}}}\leq C\norm{g}_{H^{\frac{n}{n+\alpha}}_\L}$.
\end{lem}

\begin{proof}
Let $g$ be an $H^{\frac{n}{n+\alpha}}_\L$--atom associated to a ball $B=B(x_0,r)$. We apply H\"older's inequality and the $L^2$-boundedness of the area function \eqref{area function} to get
\begin{align*}
    \int_{8B}\abs{S_\beta(g)(x)}^{\frac{n}{n+\alpha}}~dx
     \leq C\abs{B}^{\frac{n+2\alpha}{2(n+\alpha)}}\norm{g}_{L^2(8B)}^{\frac{n}{n+\alpha}}
     &\leq C\abs{B}^{\frac{n+2\alpha}{2(n+\alpha)}}\abs{B}^{\frac{n}{2(n+\alpha)}}
     \norm{g}_{L^\infty}^{\frac{n}{n+\alpha}}\leq C.
\end{align*}

In order to complete the proof of Lemma \ref{lem:est S}, we must
find a uniform bound for
\begin{equation}\label{need uniform}
\int_{(8B)^c}|S_{\beta}(g)(x)|^{\frac{n}{n+\alpha}}~dx.
\end{equation}
Let us consider first the case when $r<\tfrac{\rho(x_0)}{4}$. Then, by the moment condition on $g$, we have
\begin{align*}
\big(S_{\beta}&(g)(x)\big)^2=\int_0^\infty
         \int_{|x-y|<t}\left(\int_{\mathbb{R}^n}\left(t^\beta\partial_t^\beta\P_t(y,
         x')-t^\beta\partial_t^\beta\P_t(y, x_0)\right)g(x')~dx'\right)^2~
         \frac{dy~dt}{t^{n+1}}\\
  &\le \int_0^{\frac{|x-x_0|}{2}}
         \int_{|x-y|<t}\left(\int_{B}|t^\beta\partial_t^\beta\P_t(y,
         x')-t^\beta\partial_t^\beta\P_t(y, x_0)|~\frac{dx'}{|B|^{\frac{n+\alpha}{n}}}\right)^2~\frac{dy~dt}{t^{n+1}}\\
  &\quad    +\int_{\frac{|x-x_0|}{2}}^\infty \int_{|x-y|<t}\left(\int_{B}|t^\beta\partial_t^\beta\P_t(y,
         x')-t^\beta\partial_t^\beta\P_t(y, x_0)|~\frac{dx'}{|B|^{\frac{n+\alpha}{n}}}\right)^2~\frac{dy~dt}{t^{n+1}}=: I_1(x)+I_2(x).
\end{align*}
We now use the smoothness of $t^\beta\partial_t^\beta\P_t(y,
x)=t^\beta\partial_t^\beta\P_t(x, y)$ established in Proposition \ref{Prop:Poisson est}\textit{(c)}
 with $\alpha<\delta'<\beta$ and $N>0$. In the domain of integration of $I_1(x)$
  we have $\abs{x-x_0}\leq 2\abs{y-x_0}$. So
  \begin{align*}
I_1(x)&\leq C\int_0^{\frac{|x-x_0|}{2}}
         \int_{|x-y|<t}\left(\int_{B}\left(\frac{|x'-x_0|}{t}\right)^{\delta'}
         \frac{t^\beta}{(\abs{x_0-y}^2+t^2)^{\frac{n+\beta}{2}}}~
         \frac{dx'}{|B|^{\frac{n+\alpha}{n}}}\right)^2~
         \frac{dy~dt}{t^{n+1}}\nonumber\\
  &\leq C\int_0^{\frac{|x-x_0|}{2}}
         \int_{|x-y|<t}\left(\frac{r}{t}\right)^{2\delta'}
         \frac{1}{t^{2n}\left(\frac{\abs{x_0-y}}{t}+1\right)^{2(n+\beta)}}~
         \frac{1}{|B|^{\frac{2\alpha}{n}}}~
         \frac{dy~dt}{t^{n+1}}\nonumber\\
  &\leq C\int_0^{\frac{|x-x_0|}{2}}
         \left(\frac{r}{t}\right)^{2\delta'}
         \frac{1}{t^{2n}\left(\frac{\abs{x_0-x}}{t}\right)^{2(n+\beta)}}~
         \frac{1}{|B|^{\frac{2\alpha}{n}}}~
         \frac{dt}{t}\\
   &\leq C\frac{r^{2(\delta'-{\alpha})}}{|x-x_0|^{2(n+\beta)}}\int_0^{\frac{|x-x_0|}{2}}
          t^{2(\beta-\delta')}~\frac{dt}{t}= C\frac{r^{2(\delta'-{\alpha})}}{|x-x_0|^{2(n+\delta')}}.\nonumber
\end{align*}
Thus, integrating over $(8B)^c$, we have $\displaystyle\int_{(8B)^c}|I_1(x)^{1/2}|^{\frac{n}{n+\alpha}}~dx\leq C\int_{(8B)^c}\left(\frac{r^{\delta'-{\alpha}}}{|x-x_0|^{n+\delta'}}\right)^{\frac{n}{n+\alpha}}~dx=C$. Let us continue with $I_2(x)$. If $x\in(8B)^c$ then we have
$\displaystyle |x'-x_0|\leq r<\frac{|x-x_0|}{2}\leq t$. Then, by
Proposition \ref{Prop:Poisson est}\textit{(c)} and $\displaystyle
x\in \left(8B\right)^c$, we have
\begin{align*}
I_2(x)&\leq C\int_{\frac{|x-x_0|}{2}}^\infty
         \int_{|x-y|<t}\left(\int_{B}\left(\frac{|x'-x_0|}{t}\right)^{\delta'}\frac{1}{t^n}~\frac{dx'}{|B|^{\frac{n+\alpha}{n}}}\right)^2~
         \frac{dy~dt}{t^{n+1}}\\
  &\leq C\int_{\frac{|x-x_0|}{2}}^\infty
         \int_{|x-y|<t}\left(\frac{r}{t}\right)^{2\delta'}
         ~\frac{1}{t^{2n}}~\frac{1}{\abs{B}^{\frac{2\alpha}{n}}}~\frac{dy~dt}{t^{n+1}}= C\frac{r^{2(\delta'-\alpha)}}{|x-x_0|^{2(n+\delta')}}.
\end{align*}
Therefore the integral of $|\left(I_2(x)\right)^{1/2}|^{\frac{n}{n+\alpha}}$ over $(8B)^c$ is bounded by a constant. Collecting terms we see that if $r<\frac{\rho(x_0)}{4}$ then a
uniform bound for \eqref{need uniform} is obtained.

We now turn the the estimate of \eqref{need uniform} when $r$ is
comparable to $\rho(x_0)$, namely,
$\tfrac{\rho(x_0)}{4}<r\leq\rho(x_0)$. For $x\in(8B)^c$ we can split
the integral in $t>0$ in the definition of $\displaystyle
S_{\beta}g(x)$ into three parts:
\begin{align*}
\left(S_{\beta}(g)(x)\right)^2
  &=\left(\int_0^{\frac{r}{2}}~+\int_{\frac{r}{2}}^{\frac{|x-x_0|}{4}}~+\int_{\frac{|x-x_0|}{4}}^\infty~\right)
         \int_{|x-y|<t}\abs{\int_{\mathbb{R}^n}t^\beta\partial_t^\beta\P_t(y,
         x')g(x')~dx'}^2~\frac{dy~dt}{t^{n+1}}\\
  &=: I_1'(x)+I_2'(x)+I_3'(x).
\end{align*}
In the integrand of $I_1'(x)$,  we have $|x'-y|\sim |x-x_0|$.  So by
Proposition \ref{Prop:Poisson est}\textit{(b)}, we get
\begin{align*}
I_1'(x)
  &\le C\int_0^{\frac{r}{2}}\int_{|x-y|<t}\left(\int_{B}\frac{t^\beta}{(|y-x'|+t)^{n+\beta}}~
         \frac{1}{|B|^{\frac{n+\alpha}{n}}}~dx'\right)^2~\frac{dy~dt}{t^{n+1}}\\
  &\le C r^{-2\alpha}\int_0^{\frac{r}{2}}\int_{|x-y|<t}\frac{t^{2\beta}}{(|x-x_0|+t)^{2(n+\beta)}}
         ~\frac{dy~dt}{t^{n+1}}\le C \frac{r^{2(\beta-\alpha)}}{|x-x_0|^{2(n+\beta)}}.
\end{align*}
For $I_2'(x)$, by applying Proposition \ref{Prop:Poisson
est}\textit{(b)} for any $M>\alpha$, together with
{$|x'-y|\sim|x-x_0|$} and $\rho(x')\sim \rho(x_0)\sim r$, we get
\begin{align*}
I_2'(x) &\le
         C\int_{\frac{r}{2}}^{\frac{|x-x_0|}{4}}\int_{|x-y|<t}\left(\int_{B}\frac{t^\beta}{(|y-x'|+t)^{n+\beta}}
         \left(\frac{\rho(x')}{t}\right)^M\frac{1}{|B|^{\frac{n+\alpha}{n}}}~dx'\right)^2~\frac{dy~dt}{t^{n+1}}\\
     &\leq C\int_{\frac{r}{2}}^{\frac{|x-x_0|}{4}}\int_{|x-y|<t}\left(\int_{B}\frac{1}{t^n\left(\frac{|x-x_0|}{t}+1\right)^{n+\beta}}
         \left(\frac{\rho(x_0)}{t}\right)^{M}\frac{1}{|B|^{\frac{n+\alpha}{n}}}~dx'\right)^2\frac{dy~dt}{t^{n+1}}\\
     &\leq C\int_{\frac{r}{2}}^{\frac{|x-x_0|}{4}}\int_{|x-y|<t}\left(\frac{t^{\beta-M}\rho(x_0)^{M}}
          {|x-x_0|^{n+\beta}r^{\alpha}}\right)^2~\frac{dy~dt}{t^{n+1}}\leq C\int_{\frac{r}{2}}^{\frac{|x-x_0|}{4}}\left(\frac{t^{\beta-M}r^{M-\alpha}}
          {|x-x_0|^{n+\beta}}\right)^2~\frac{dt}{t}\\
     &\leq C  \frac{r^{2(\beta-\alpha)}}{|x-x_0|^{2(n+\beta)}}\int_{1}^{\frac{|x-x_0|}{2r}}{u^{2(\beta-M)}}~\frac{du}{u}
     \leq C\frac{r^{2(M-\alpha)}}{|x-x_0|^{2(n+M)}}.
\end{align*}

Finally, for the last term above $I_3'(x)$, with the same method that was used to estimate $I_2'(x)$, we obtain $I_3'(x)\leq Cr^{2(M-\alpha)}|x-x_0|^{-2(n+M)}$. Hence, $\displaystyle\int_{(8B)^c}|I_j'(x)^{1/2}|^{\frac{n}{n+\alpha}}~dx\leq C$, for $j=1,2,3$ and the uniform bound for \eqref{need uniform} is established also when $r\sim\rho(x_0)$. The proof of Lemma \ref{lem:est S} is complete.
\end{proof}

Now the three steps of the proof of \textit{(iii)} $\Longrightarrow$
\textit{(i)} in  Theorem \ref{Thm:Characterization} are completed.
It only remains to prove Lemma \ref{Lem:reproduce}, that we took for
granted before. To that end, we need the following result.

\begin{lem}\label{lem1}
Let $q_t(x,y)$ be a function of $x,y\in\Real^n$, $t>0$. Assume that
for each $N>0$ there exists a constant $C_N$ such that, for some $\gamma\geq\alpha$,
\begin{equation}\label{equ2}
|q_t(x,y)|\leq
C_N\left(1+\frac{t}{\rho(x)}+\frac{t}{\rho(y)}\right)^{-N}t^{-n}\left(1+\frac{|x-y|}{t}\right)^{-(n+\gamma)}.
\end{equation}
Then, for every $\displaystyle H_{\L}^{\frac{n}{n+\alpha}}$--atom $g$ supported on $B(x_0,r)$, there exists $C_{N,x_0,r}>0$ such that
$$\sup\limits_{t>0}\abs{\int_{\Real^n}q_t(x,y)g(y)~dy}\leq C_{N,x_0,r}\left(1+|x|\right)^{-(n+\gamma)},\quad x\in\Real^n.$$
\end{lem}

\begin{proof}
Let $I=I(x,t)$ be the integral appearing in the statement. If $x\in B(x_0,2r)$ then, since $\norm{g}_{L^\infty(\Real^n)}\leq\abs{B(x_0,r)}^{-(1+\frac{\alpha}{n})}$, we have 
$$\abs{I} \le C_{N}\frac{1}{r^{n+\alpha}}\int_{\Real^n}t^{-n}\left(1+\frac{|x-y|}{t}\right)^{-(n+\gamma)}~dy \le  C_{N}\frac{1}{r^{n+\alpha}}\int_{\mathbb{R}^n}\frac{1}{\left(1+|u|\right)^{n+\gamma}}~du \le C_{N,r}.$$
Since $|x-x_0|\leq 2r$, we have $1+|x|\leq 1+|x-x_0|+|x_0|\leq 1+2r+|x_0|$. Hence $\left|I\right|\leq C_{N,r}\frac{(1+2r+|x_0|)^{n+\gamma}}{(1+2r+|x_0|)^{n+\gamma}}\leq C_{N,x_0,r}(1+|x|)^{-(n+\gamma)}$. If $x\notin B(x_0,2r)$ then for $y\in B(x_0,r)$ we have $|x-y|\sim|x-x_0|$ and, since $r<\rho(x_0)$, we get that $\rho(x_0)\sim\rho(y)$, see Lemma \ref{Lem:equiv rho}. Hence, choosing $N=\gamma$ in \eqref{equ2}, we get
\begin{align*}
\abs{I} &\leq
C_\gamma\frac{t^{-\gamma}t^{-n}\abs{x-x_0}^{-(n+\gamma)}}{\rho(x_0)^{-\gamma}t^{-(n+\gamma)}}\norm{g}_{L^1(\Real^n)}\leq C_{\gamma,x_0,r}\frac{\rho(x_0)^\gamma|x-x_0|^{-(n+\gamma)}}{r^\gamma}\leq C_{\gamma,x_0,r}|x-x_0|^{-(n+\gamma)}.
\end{align*}
Since $x\notin B(x_0,2r)$, we can set $x=x_0+2rz$, $|z|\geq 1$. Then $1+|x|\leq1+|x_0|+2r|z|$, and
$\frac{1+|x_0|+2r}{2r}\big|x-x_0\big|=(1+|x_0|+2r)|z|\geq
1+|x_0|+2r|z|$. It means that $c_{x_0,r}|x-x_0|\geq 1+|x|$. Therefore $\abs{I}\leq C_{\gamma,x_0,r}|x-x_0|^{-(n+\gamma)}\leq C_{\gamma,x_0,r}(1+|x|)^{-(n+\gamma)}$.
\end{proof}

\begin{proof}[Proof of Lemma \ref{Lem:reproduce}]
Assume that $g$ is an $\displaystyle
H^{\frac{n}{n+\alpha}}_\L$--atom associated to a ball $B=B(x_0,r)$.
By Lemma \ref{Lem:Pola} and Lemma \ref{lem:est S}, the following
integral is absolutely convergent and therefore it can be described
as
$$I=\int_{\Real_+^{n+1}}F(x,t)\overline{G(x,t)}~\frac{dx~dt}{t}=\lim_{\epsilon\to0}\int_\epsilon^{1/\epsilon}
\int_{\mathbb{R}^n}t^\beta\partial_t^\beta\P_tf(x)\overline{t^\beta\partial_t^\beta\P_tg(x)}~\frac{dx~dt}{t}.$$
Proposition \ref{Prop:Poisson est}\textit{(b)} and $\beta>\alpha+\varepsilon$ imply
that $q_t(x,y):=t^\beta\partial_t^\beta\P_t(x,y)$ satisfies
\eqref{equ2} in Lemma \ref{lem1}. Therefore, since $f\in
L^1(\Real^n,(1+\abs{x})^{-(n+\alpha+\varepsilon)}dx)$, Fubini's theorem can
be applied to get:
$$\int_{\Real^n}t^\beta\partial_t^\beta\P_tf(x)\overline{t^\beta\partial_t^\beta\P_tg(x)}~dx= \int_{\Real^n}f(y)\overline{(t^\beta\partial_t^\beta\P_t)^2g(y)}~dy.$$
So that,
\begin{align}\label{equ3}
I&=\lim_{\epsilon\to0}\int_\epsilon^{1/\epsilon}\left[\int_{\Real^n}f(y)
      \overline{(t^\beta\partial_t^\beta\P_t)^2g(y)}~dy\right] ~\frac{dt}{t}
   = \lim_{\epsilon\to0}\int_{\Real^n}f(y)\left[\int_\epsilon^{1/\epsilon}
       \overline{t^{2\beta}\partial_t^{2\beta}\P_{2t}g(y)} ~\frac{dt}{t}\right]~dy.
\end{align}
We claim that
\begin{equation}\label{equ1}
\sup\limits_{\epsilon>0}\abs{\int_{\epsilon}^{1/\epsilon}t^{2\beta}\partial_t^{2\beta}\P_{2t}g(y)~\frac{dt}{t}}\leq
C(1+|y|)^{-(n+\alpha+\varepsilon)},
\end{equation}
 for any $y\in \mathbb{R}^n$. To prove \eqref{equ1} we first note that
\begin{multline*}
\abs{\int_{\epsilon}^{1/\epsilon}
t^{2\beta}\partial_t^{2\beta}\P_{2t}g(y)~\frac{dt}{t}}
   \le\abs{\int_{\epsilon}^\infty t^{2\beta}\partial_t^{2\beta}\P_{2t}g(y)~
          \frac{dt}{t}}+\abs{\int_{1/\epsilon}^\infty t^{2\beta}\partial_t^{2\beta}\P_{2t}g(y)
          ~\frac{dt}{t}}\\
   =\abs{\int_{\mathbb{R}^n}\int_{\epsilon}^\infty
        t^{2\beta}\partial_t^{2\beta}\P_{2t}(x,y)~\frac{dt}{t}~g(x)~dx}
        +\abs{\int_{\mathbb{R}^n}\int_{1/\epsilon}^\infty
        t^{2\beta}\partial_t^{2\beta}\P_{2t}(x,y)~\frac{dt}{t}~g(x)~dx}.
\end{multline*}
Hence, to prove \eqref{equ1} it is enough to check that the kernel
\begin{equation}\label{crazy kernel}
\int_{\epsilon}^\infty
t^{2\beta}\partial_t^{2\beta}\P_{2t}(x,y)~\frac{dt}{t}={2^{[2\beta]-2\beta+1}}\int_{2\epsilon}^\infty
t^{2\beta}\partial_t^{2\beta}\P_{t}(x,y)~\frac{dt}{t},
\end{equation}
satisfies estimate \eqref{equ2} of Lemma \ref{lem1}, for any
$\epsilon>0$. To verify this we consider it in three cases.
\newline\noindent\textbf{Case I: $2\beta<1$.} Making a change of variables in the
definition of the fractional derivative \eqref{frac deriv}, applying
Fubini's theorem and integrating by parts, we have 
\begin{align*}
&\int_{2\epsilon}^\infty
t^{2\beta}\partial_t^{2\beta}\P_{t}(x,y)~\frac{dt}{t}=C\int_{2\epsilon}^\infty
t^{2\beta}\int_t^\infty
      \partial_u \P_u(x,y)(u-t)^{-2\beta}~du~\frac{dt}{t}\\
  &=  C\int_{2\epsilon}^\infty \partial_u \P_u(x,y)
      \int_{\frac{2\epsilon}{u}}^1
      \left(\frac{w}{1-w}\right)^{2\beta}~\frac{dw}{w}~du=C\int_{2\epsilon}^\infty\P_u(x,y) \left(\frac{2\epsilon}{u-2\epsilon}\right)^{2\beta}~\frac{du}{u}\\
  &=C\int_{2\epsilon}^\infty\P_u(x,y) \left(\frac{2\epsilon}{u-2\epsilon}\right)^{2\beta}
       \chi_{A}(u)~\frac{du}{u}
      +C\int_{2\epsilon}^\infty\P_u(x,y) \left(\frac{2\epsilon}{u-2\epsilon}\right)^{2\beta}
         \chi_{A^c}(u)~\frac{du}{u}=: I'+II',
\end{align*}
where $A=\{u-2\epsilon\le \epsilon+{|x-y|}\}$. Observe that in the
equalities above we applied the assumption $2\beta<1$ to have
convergent integrals. Let us first estimate $I'$. By Proposition
\ref{Prop:Poisson est}\textit{(a)} and since $\alpha+\varepsilon<2\beta$ we get that for any $N>0$,
\begin{align*}
\abs{I'}
  &\le C \frac{\epsilon^{2\beta}}{(\abs{x-y}+\epsilon)^{{n+1}}}\left(1+\frac{\epsilon}{\rho(x)}
          + \frac{\epsilon}{\rho(y)}\right)^{-N}\int_{2\epsilon}^{3\epsilon+\abs{x-y}}(u-2\epsilon)^{-2\beta}~du\\
  &\le C {\epsilon^{2\beta}}\left(1+\frac{\epsilon}{\rho(x)}
          + \frac{\epsilon}{\rho(y)}\right)^{-N}\left({\abs{x-y}+\epsilon}\right)^{-n-2\beta},
\end{align*}
and the desired estimate follows. We continue now with $II'$. Note that in $II'$ we have
$u-2\epsilon>|x-y|+\epsilon$ so, again by Proposition
\ref{Prop:Poisson est}\textit{(a)}, we get 
\begin{align*}
\abs{II'} &\le  C
\left(\frac{\epsilon}{\epsilon+\abs{x-y}}\right)^{2\beta}\left(1+\frac{\epsilon}{\rho(x)}
          + \frac{\epsilon}{\rho(y)}\right)^{-N}
          \int_{2\epsilon}^\infty\left(\abs{x-y}+u\right)^{-n-1} du\\
  &=  C \left(\frac{\epsilon}{\epsilon+\abs{x-y}}\right)^{2\beta}\left(1+\frac{\epsilon}{\rho(x)}
          + \frac{\epsilon}{\rho(y)}\right)^{-N}\left(\epsilon+\abs{x-y}\right)^{-n},
\end{align*}
which implies the estimate.
\newline\noindent\textbf{Case II: $2\beta=1$.} By Proposition \ref{Prop:Poisson est}\textit{(b)} and integrating by parts it is
easy to verify condition (\ref{equ2}) for $\displaystyle\int_{\epsilon}^\infty\partial_t\P_{2t}(x,y)~dt$, for any $\epsilon>0$.
\newline\noindent\textbf{Case III: $2\beta>1$.} Let $k\geq2$ be the integer such that $k-1<2\beta\leq
k$. Note that the estimate is easy when $2\beta=k,$ just integrating
by parts. When $k-1<2\beta<k$ we make a computation similar to the
case $2\beta<1$. In fact,
\begin{align}\label{kernel est}
&\int_{2\epsilon}^\infty
t^{2\beta}\partial_t^{2\beta}\P_t(x,y)~\frac{dt}{t}=
C\int_{2\epsilon}^\infty \partial_u^k \P_u(x,y) \int_{2\epsilon}^u
      t^{2\beta}(u-t)^{k-2\beta-1}~\frac{dt}{t}~du\nonumber\\
  &=  C\int_{2\epsilon}^\infty u^{k-1} \partial_u^k \P_u(x,y)
      \int_{\frac{2\epsilon}{u}}^1
      w^{2\beta}(1-w)^{k-2\beta-1}~\frac{dw}{w}~du\nonumber\\
  &= C\int_{2\epsilon}^\infty
  u^{k-1}\partial_u^{k-1}\mathcal{P}_u(x,y)\frac{(2\epsilon)^{2\beta}u^{1-k}}{(u-2\epsilon)^{1+2\beta-k}}~\frac{du}{u}+C\int_{2\epsilon}^\infty u^{k-2}\partial_u^{k-2}\mathcal{P}_u(x,y)\frac{(2\epsilon)^{2\beta}u^{1-k}}
  {(u-2\epsilon)^{1+2\beta-k}}~\frac{du}{u}\\
  &\quad +\cdots+C\int_{2\epsilon}^\infty u\partial_u\mathcal{P}_u(x,y)\frac{(2\epsilon)^{2\beta}u^{1-k}}{(u-2\epsilon)^{1+2\beta-k}}~\frac{du}{u}+C\int_{2\epsilon}^\infty
  \mathcal{P}_u(x,y)\frac{(2\epsilon)^{2\beta}u^{1-k}}{(u-2\epsilon)^{1+2\beta-k}}~\frac{du}{u}.\nonumber
 \end{align}
For any $1\le m\le k-1$ apply Proposition \ref{Prop:Poisson est}\textit{(b)} to get that for any $N>0$
\begin{align*}
&\abs{\int_{2\epsilon}^\infty
  u^{m}\partial_u^{m}\mathcal{P}_u(x,y)\frac{(2\epsilon)^{2\beta}u^{1-k}}{(u-2\epsilon)^{1+2\beta-k}}~\frac{du}{u}}\\
  &\le C\frac{\epsilon^{2\beta}}{(\epsilon+\abs{x-y})^{n+m}}\left(1+\frac{\epsilon}{\rho(x)}+ \frac{\epsilon}{\rho(y)}\right)^{-N}
  \int_{2\epsilon}^\infty {(u-2\epsilon)^{k-2\beta-1}}~\frac{du}{u^{k-m}}\\
  &=C\frac{\epsilon^{2\beta}}{(\epsilon+\abs{x-y})^{n+m}}\left(1+\frac{\epsilon}{\rho(x)}+ \frac{\epsilon}{\rho(y)}\right)^{-N}
     \int_{2\epsilon}^{3\epsilon}(u-2\epsilon)^{k-2\beta-1}~\frac{du}{u^{k-m}}\\
   &\quad +C\frac{\epsilon^{2\beta}}{(\epsilon+\abs{x-y})^{n+m}}
   \left(1+\frac{\epsilon}{\rho(x)}+ \frac{\epsilon}{\rho(y)}\right)^{-N}\int_{3\epsilon}^\infty
   (u-2\epsilon)^{k-2\beta-1}~\frac{du}{u^{k-m}}=: I''+II''.
\end{align*}
For $I''$, since $2\beta< k$ and $m\ge 1>\alpha+\varepsilon$, we
obtain
\begin{align*}
I''&\le
C\frac{\epsilon^{m}}{(\epsilon+\abs{x-y})^{n+m}}\left(1+\frac{\epsilon}{\rho(x)}+
    \frac{\epsilon}{\rho(y)}\right)^{-N}\\
  &\le C\frac{1}{(\epsilon+\abs{x-y})^{n}}\left(1+\frac{\epsilon}{\rho(x)}+
    \frac{\epsilon}{\rho(y)}\right)^{-N}\left(\frac{\epsilon}{\epsilon+\abs{x-y}}\right)^{\alpha+\varepsilon},
\end{align*}
and the estimate follows. For $II''$, since $\tfrac{1}{u}<\frac{1}{u-2\epsilon}$ and
$m<2\beta$, we also have
\begin{align*}
II''  &\le C\frac{\epsilon^{m}}{(\epsilon+\abs{x-y})^{n+m}}
   \left(1+\frac{\epsilon}{\rho(x)}+ \frac{\epsilon}{\rho(y)}\right)^{-N},
\end{align*}
which gives the bound. For the last term of \eqref{kernel est} we get an estimate as
above by Proposition \ref{Prop:Poisson est}\textit{(b)}.

Hence, from the three cases above we see that the kernel
\eqref{crazy kernel} satisfies condition \eqref{equ2} in Lemma
\ref{lem1}, for any $\epsilon>0$. Therefore can pass the limit
inside the integral in \eqref{equ3}. Then, by Lemma
\ref{Lem:equality L2}, we have
$$I=\frac{4^{\beta}}{\Gamma(2\beta)}\int_{\mathbb{R}^n}f(y)\overline{g(y)}~dy.$$
This establishes Lemma \ref{Lem:reproduce} and it finally completes
the proof of \textit{(iii)} $\Longrightarrow$ \textit{(i)}.
\end{proof}

\subsection{Proof of Theorem \ref{Thm:alpha 1}\textit{(II)}}

Let us begin with the following

\begin{prop}\label{Prop:point Lambda}
Let $0<\alpha\le 1$ and $f$ be a function in $L^\infty(\Real^n)$ such that $\abs{f(x)}\leq C\rho(x)^\alpha$, for
some constant $C$ and all $x\in\Real^n$. Then $\|t^\beta\partial_t^\beta\P_tf\|_{L^\infty(\Real^n)}\leq Ct^\alpha$, for any $\beta>\alpha$, if and only if $\abs{f(x+y)+f(x-y)-2f(x)}\leq C\abs{y}^\alpha$, for all $x,y\in\Real^n$.
\end{prop}

Let us show how this proposition can be applied to prove Theorem
\ref{Thm:alpha 1}\textit{(II)} first.

\begin{proof}[Proof of Theorem \ref{Thm:alpha 1}\textit{(II)}]
Assume first $n=1$. Consider the function, see \cite[p.~148]{SteinSingular}, $f(x)=\sum_{k=1}^\infty2^{-k}e^{2\pi i2^kx}$, $x\in\Real$. Observe that $\rho(x)\equiv\tfrac{1}{\sqrt{\mu}}$. Therefore there exists a constant $C=2\sqrt{\mu}$
such that $\abs{f(x)}\le\sum_{k=1}^\infty 2^{-k}=1\le\tfrac{C}{\sqrt{\mu}}=C\rho(x)$, for all $x\in\Real$. Now, for any $y\in\Real$,
$$f(x+y)+f(x-y)-2f(x)=2\sum_{k=1}^\infty 2^{-k}\big(\cos(2\pi2^ky)-1\big)e^{2\pi i 2^k x}.$$
Since $\abs{\cos(2\pi 2^ky)-1}\le C(2^ky)^2$ and $\abs{\cos(2\pi
2^ky)-1}\le 2$, we have
$$\abs{f(x+y)+f(x-y)-2f(x)}\le C\sum_{2^k\abs{y}\le 1}2^{-k}(2^ky)^2+C\sum_{2^k\abs{y}>1}2^{-k}\leq C\abs{y}.$$
So, by Proposition \ref{Prop:point Lambda}, we obtain  
$\|t^\beta\partial_t^\beta\P_tf\|_{L^\infty(\Real^n)}\leq Ct$. Let us see
that $f$ can not be a function in $C^{0,1}_{\L_\mu}$. To arrive to a
contradiction suppose that $\abs{f(x+y)-f(x)}\le C_f\abs{y}$, for
any $x,y \in\Real$. Then by Bessel's inequality for $L^2$ periodic
functions we would have
$$(C_f\abs{y})^2\ge\int_0^1\abs{f(x+y)-f(x)}^2~dx=\sum_{k=1}^\infty2^{-2k}|e^{2\pi i2^ky}-1|^2
\ge \abs{y}^2\sum_{2^k\abs{y}\le1}|e^{2\pi i2^ky}-1|^2.$$
Note that in the range $2^k\abs{y}\le 1$ we have $|e^{2\pi i2^ky}-1|^2\ge c(2^ky)^2$. Hence we arrive to the contradiction $C_f^2\ge c\abs{y}^2\sum_{2^k|y|\le 1}2^{2k}$.

For the case $n\ge 2$, note that we can write
$\L_\mu=\L^1_\mu-\frac{\partial^2}{\partial{{x_2}^2}}-\cdots-\frac{\partial^2}{\partial{{x_n}^2}}$,
where $\L^1_\mu=-\frac{\partial^2}{\partial{{x_1}^2}}+\mu$. The
operator $\L^1_\mu$ acts only in the one dimensional variable $x_1$.
Let us define $g(x_1,\ldots,x_n)=f(x_1)$, with $f$ as above. Then,
with an easy computation using the subordination formula
\eqref{subordinacion}, we have
$\|t^\beta\partial_t^\beta\P_tg\|_{L^\infty(\Real^n)}=\|t^\beta\partial_t^\beta
e^{-t\sqrt{\L^1_\mu}}f\|_{L^\infty(\Real)}\le Ct$, and, for any
$x,x'\in \Real^n,$ the inequality
$\abs{g(x)-g(x')}=\abs{f(x_1)-f(x_1')}\le C\abs{x_1-x_1'}\le
C\abs{x-x'}$ fails for any $C>0$.
\end{proof}

To prove Proposition \ref{Prop:point Lambda} we need the following two lemmas.

\begin{lem}\label{Lem:equivalence}
Let $f$ be a locally integrable function on $\Real^n$, $n\geq 3$,
and $\alpha>0$. If there exists $\beta>\alpha$ such that $\|t^\beta\partial_t^\beta\P_tf\|_{L^\infty(\mathbb{R}^n)}\leq C_\beta t^\alpha$, for all $t>0$, then for any $\sigma>\alpha$ we also have $\|t^\sigma\partial_t^\sigma\P_tf\|_{L^\infty(\mathbb{R}^n)}\leq C_\sigma
t^\alpha$, for all $t>0$. Moreover, the constants $C_\beta$ and $C_\sigma$ are comparable.
\end{lem}

\begin{proof}
Assume first that $\sigma>\beta>\alpha$. Then, by hypothesis and Proposition \ref{Prop:Poisson est}\textit{(b)}, we have
\begin{align*}
    |t^\sigma\partial_t^\sigma\P_tf(x)| &=
    |t^\sigma\partial_t^{\sigma-\beta}\P_{t/2}(\partial_t^\beta\P_{t/2}f)(x)|=t^\sigma\abs{\int_{\Real^n}\partial_t^{\sigma-\beta}\P_{t/2}(x,y)
    \partial_t^\beta\P_{t/2}f(y)~dy} \\
     &\leq      Ct^{\sigma+\alpha-\beta}\int_{\Real^n}\frac{1}{(\abs{y}+t)^{n+\sigma-\beta}}~dy= Ct^\alpha.
\end{align*}

Suppose now that $\alpha<\sigma<\beta$. Let $k$ be the least
positive integer for which $\sigma<\beta\leq\sigma+k$. Applying the
case just proved above, we get 
\begin{align*}
    |t^\sigma\partial_t^\sigma\P_tf(x)| &\leq
    t^\sigma\int_t^\infty\int_{s_1}^\infty\cdots\int_{s_{k-1}}^\infty\abs{\partial_{s_k}^{k+\sigma}\P_{s_k}f(x)}~ds_{k}~\cdots~ds_2~ds_1 \\
    &\leq Ct^\sigma\int_t^\infty\int_{s_1}^\infty\cdots\int_{s_{k-1}}^\infty
    s_k^{\alpha-(k+\sigma)}ds_{k}~\cdots~ds_2~ds_1= Ct^\alpha.
\end{align*}
\end{proof}

\begin{lem}\label{lem:5.1}
Let $0<\alpha\le 1$. If a function $f$ satisfies $\abs{f(x)}\leq
C\rho(x)^\alpha$ for all $x\in\Real^n$ then for any $\beta>\alpha$, $\|t^\beta\partial_t^\beta(\P_t-P_t)f\|_{L^\infty(\mathbb{R}^n)}\leq Ct^\alpha$, for all $t>0$, where $P_t$ is the classical Poisson semigroup \eqref{Classical Poisson} with kernel \eqref{classical Poisson kernel}.
\end{lem}

\begin{proof}
Let $\beta>\alpha$ and $m=[\beta]+1$. In a parallel way as in \eqref{deriv Poisson}, we can derive a formula for the kernel $D_\beta(x,y,t)$ of the operator $t^\beta\partial_t^\beta(\P_t-P_t)$ in terms of the heat kernels for $\L$ and $-\Delta$ given in \eqref{heatL} and \eqref{classical heat}:
\begin{align*}
 &D_\beta(x,y,t)=t^\beta\partial_t^\beta\int_0^\infty\frac{te^{-\frac{t^2}{4u}}}{2\sqrt{\pi}}~(k_u(x,y)-h_u(x-y))~\frac{du}{u^{3/2}} \\
 &= Ct^\beta\int_0^\infty\int_0^\infty H_{m+1}\left(\frac{t+s}{2\sqrt{u}}
 \right)e^{-\frac{(t+s)^2}{4u}} \left(\frac{1}{\sqrt{u}}\right)^{m+1}s^{m-\beta}~\frac{ds}{s}~(k_u(x,y)-h_u(x-y))~\frac{du}{u^{1/2}}.
\end{align*}
Then, by Lemma \ref{Lem:Schwartz}, we have
\begin{align*}
 \abs{D_\beta(x,y,t)}
 &\le Ct^\beta\int_0^\infty\int_0^\infty e^{-c\frac{(t+s)^2}{4u}}
     \left(\frac{1}{\sqrt{u}}\right)^{m+1}s^{m-\beta}~\frac{ds}{s}\abs{k_u(x,y)-h_u(x-y)}~\frac{du}{u^{1/2}}\\
 &\le C\int_0^\infty e^{-c\frac{t^2}{4u}}
     \left(\frac{t}{\sqrt{u}}\right)^{\beta}\left(\frac{\sqrt{u}}{\rho(y)}\right)^\alpha~w_u(x-y)~\frac{du}{u},
 \end{align*}
where the function $w\in\mathcal{S}$ is nonnegative. Hence, for all  $x\in \mathbb{R}^n$,
\begin{align*}
|t^\beta\partial_t^\beta(\P_t-P_t)f(x)|  &\le  C\int_{\mathbb{R}^n}\int_0^\infty e^{-c\frac{t^2}{4u}}
        \left(\frac{t}{\sqrt{u}}\right)^{\beta} \left(\frac{\sqrt{u}}{\rho(y)}\right)^\alpha~w_u(x-y)~\frac{du}{u}~\rho(y)^\alpha~dy\\
  &\le C\int_0^\infty e^{-c\frac{t^2}{4u}} \left(\frac{t}{\sqrt{u}}\right)^{\beta}
        \left({\sqrt{u}}\right)^\alpha~\frac{du}{u}= Ct^\alpha\int_0^\infty e^{-v}v^{\frac{\beta-\alpha}{2}}~\frac{dv}{v}=C t^\alpha.
\end{align*}
\end{proof}

\begin{proof}[Proof of Proposition \ref{Prop:point Lambda}]
Assume that $\|t^\beta\partial_t^\beta\P_tf\|_{L^\infty(\mathbb{R}^n)}\leq C
t^\alpha,\ \hbox{for any}~\beta>\alpha$. Then, by Lemma \ref{lem:5.1}, we obtain $\|t^\beta\partial_t^\beta P_tf\|_{L^\infty(\mathbb{R}^n)}\le
\|t^\beta\partial_t^\beta(P_t-\P_t)f\|_{L^\infty(\mathbb{R}^n)}+\|t^\beta\partial_t^\beta\P_tf\|_{L^\infty(\mathbb{R}^n)}
\leq Ct^\alpha$. Therefore, as $f$ is bounded, $f$ is in the classical $\alpha$-Lipschitz space $\Lambda^\alpha$, see \cite{SteinSingular}. Hence $\abs{f(x+y)+f(x-y)-2f(x)}\leq C\abs{y}^\alpha$, for all $x,y\in\Real^n$.

For the converse, since $f\in L^\infty(\Real^n)$, then, by \cite{SteinSingular}, $\|t^2\partial_t^2 P_tf\|_{L^\infty(\mathbb{R}^n)}\le Ct^\alpha.$ So Lemma \ref{lem:5.1} gives $\|t^2\partial_t^2 \P_tf\|_{L^\infty(\mathbb{R}^n)}\le \|t^2\partial_t^2 (\P_t-P_t)f\|_{L^\infty(\mathbb{R}^n)}
+\|t^2\partial_t^2 P_tf\|_{L^\infty(\mathbb{R}^n)}\le Ct^\alpha$. Thus, by Lemma \ref{Lem:equivalence}, we get $\|t^\beta\partial_t^\beta\P_tf\|_{L^\infty(\mathbb{R}^n)}\leq Ct^\alpha$ for any $\beta>\alpha$.
\end{proof}

\subsection{Proof of Theorem \ref{Thm:alpha 0}\textit{(A)}}

As explained at the beginning of this section, we only need to prove the necessity part. Let $f\in BMO_\L$. Let us fix
a ball $B=B(x_0,r)$ and write $f=f_1+f_2+f_3$, with $f_1=(f-f_B)\chi_{2B}$, $f_2=(f-f_B)\chi_{(2B)^c}$ and $f_3=f_B$.

For $f_1$, by the boundedness of the area function \eqref{area
function} on $L^2(\Real^n)$ and Remark \ref{Rem:equiv norm p} with $p=2$,
\begin{align*}
    \frac{1}{\abs{B}}&\int_{\widehat{B}}|t^\beta\partial_t^\beta\P_tf_1(x)|^2~\frac{dx~dt}{t} = \frac{1}{\abs{B}}
    \int_{\widehat{B}}|t^\beta\partial_t^\beta\P_tf_1(x)|^2\int_{\Real^n}\chi_{\abs{x-z}<t}(z)~dz~\frac{dx~dt}{t^{n+1}} \\
     &\leq \frac{1}{\abs{B}}\int_{\abs{x_0-z}<2r}\int_0^\infty\int_{\Real^n}|t^\beta\partial_t^\beta\P_tf_1(x)|^2
     \chi_{\abs{x-z}<t}(z)~\frac{dx~dt}{t^{n+1}}~dz \\
     &= \frac{1}{\abs{B}}\int_{\abs{x_0-z}<2r}\iint_{\mathbf{\Gamma}(z)}|t^\beta\partial_t^\beta\P_tf_1(x)|^2~\frac{dx~dt}{t^{n+1}}~dz\leq \frac{C}{\abs{B}}\int_{2B}\abs{f(z)-f_B}^2~dz\leq C\norm{f}_{BMO_\L}^2.
\end{align*}

For $f_2$ and $x\in{B}$, apply Proposition \ref{Prop:Poisson est}\textit{(b)} and the classical annuli argument to get
\begin{align*}
    |t^\beta\partial_t^\beta\P_tf_2(x)| &\leq C\sum_{k=2}^\infty\int_{2^kB\setminus 2^{k-1}B}\abs{f(z)-f_{2^kB}}\frac{t^\beta}{(t+\abs{x-z})^{n+\beta}}~dz \\
     &\quad +\sum_{k=2}^\infty\sum_{j=1}^k\abs{f_{2^jB}-f_{2^{j-1}B}}\int_{2^kB
            \setminus 2^{k-1}B}\frac{t^\beta}{(t+\abs{x-z})^{n+\beta}}~dz \\
     &\leq C\left(\frac{t}{r}\right)^\beta\left(\sum_{k=2}^\infty\frac{1}{2^{k\beta}}~\frac{1}{(2^kr)^n}\int_{2^kB}
              \abs{f(z)-f_{2^kB}}~dz+\norm{f}_{BMO_\L}\sum_{k=2}^\infty
          \frac{k}{2^{k\beta}}\right) \\
     &\leq C\left(\frac{t}{r}\right)^\beta\norm{f}_{BMO_\L}\sum_{k=2}^\infty
        \frac{1+k}{2^{k\beta}}=C\left(\frac{t}{r}\right)^\beta\norm{f}_{BMO_\L}.
\end{align*}
Therefore $\displaystyle\frac{1}{\abs{B}}\int_{\widehat{B}}|t^\beta\partial_t^\beta\P_tf_2(x)|^2~\frac{dx~dt}{t}
\leq C\norm{f}_{BMO_\L}^2\int_0^r\left(\frac{t}{r}\right)^{2\beta}~\frac{dt}{t}
=C\norm{f}_{BMO_\L}^2$.

Let us finally consider $f_3$. Assume that $r\geq\rho(x_0)$. By
Proposition \ref{Prop:Poisson est}\textit{(d)}, for some
$0<\delta'\leq\delta$ with $\delta'<\beta$, we have 
$$|t^\beta\partial_t^\beta\P_tf_3(x)|\leq C\abs{f_B}\frac{(t/\rho(x))^{\delta'}}{(1+t/\rho(x))^N}
\leq C\norm{f}_{BMO_\L}\frac{(t/\rho(x))^{\delta'}}{(1+t/\rho(x))^N}.$$
Hence
\begin{align}
    \nonumber \frac{1}{\abs{B}}\int_{\widehat{B}}|t^\beta\partial_t^\beta\P_tf_3(x)|^2~\frac{dx~dt}{t}
     &\leq C\norm{f}_{BMO_\L}^2\frac{1}{\abs{B}}\int_{\widehat{B}}
     \frac{(t/\rho(x))^{2\delta'}}{(1+t/\rho(x))^{2N}}~\frac{dx~dt}{t} \\
    \label{ff3} &\leq C\norm{f}_{BMO_\L}^2\frac{1}{\abs{B}}
    \int_B\left(\int_0^{\rho(x)}~+\int_{\rho(x)}^\infty~\right) \frac{(t/\rho(x))^{2\delta'}}{(1+t/\rho(x))^{2N}}~\frac{dt}{t}~dx.
\end{align}
On one hand,
$$\int_0^{\rho(x)}\frac{(t/\rho(x))^{2\delta'}}{(1+t/\rho(x))^{2N}}~\frac{dt}{t}\leq\int_0^{\rho(x)}(t/\rho(x))^{2\delta'}~\frac{dt}{t}=C.$$
On the other hand,
$$\int_{\rho(x)}^\infty\frac{(t/\rho(x))^{2\delta'}}{(1+t/\rho(x))^{2N}}~\frac{dt}{t}
\leq\int_{\rho(x)}^\infty(t/\rho(x))^{2\delta'-2N}~\frac{dt}{t}=C.$$
Therefore from \eqref{ff3} we obtain that if $r\geq\rho(x_0)$ then $\displaystyle\frac{1}{\abs{B}}\int_{\widehat{B}}|t^\beta\partial_t^\beta\P_tf_3(x)|^2~\frac{dx~dt}{t}
\leq C\norm{f}_{BMO_\L}^2$. Suppose that $r<\rho(x_0)$. By
Remark \ref{rem:3.4}, Proposition \ref{Prop:Poisson est}\textit{(d)} with some $\delta'>1/2$
and Lemma \ref{Lem:equiv rho}, we get 
\begin{align*}
    \frac{1}{\abs{B}}\int_{\widehat{B}}|t^\beta\partial_t^\beta\P_tf_3(x)|^2~\frac{dx~dt}{t}
    &\leq C\norm{f}_{BMO_\L}^2\left(1+\log\frac{\rho(x_0)}{r}\right)^2\frac{1}{\abs{B}}
            \int_{\widehat{B}}\frac{(t/\rho(x))^{2\delta'}}{(1+t/\rho(x))^{2N}}~\frac{dx~dt}{t} \\
    &\leq C\norm{f}_{BMO_\L}^2\left(1+\log\frac{\rho(x_0)}{r}\right)^2\frac{1}{\abs{B}}
            \int_B\int_0^r(t/\rho(x_0))^{2\delta'}~\frac{dt}{t}~dx \\
    &= C\norm{f}_{BMO_\L}^2\left(1+\log\frac{\rho(x_0)}{r}\right)^2\left(\frac{r}{\rho(x_0)}\right)^{2\delta'}\le C\norm{f}_{BMO_\L}^2,
\end{align*}
for all $r<\rho(x_0)$. This finishes the proof.

\subsection{Proof of Theorem \ref{Thm:alpha 0}\textit{(B)}}
As in the argument of the proof of Theorem \ref{Thm:alpha
1}\textit{(II)}, we only need to consider the case $n=1$. We will take $\beta=1$. Let $f(x)=\max\set{\log\frac{1}{\abs{x}},0}$, $x\in\Real$. It is well known that $f$ belongs to the classical $BMO(\Real)$. Observe that the function $f$ is nonnegative and it is supported in $[-1,1]$. For every $x$ we have $\rho(x)=\frac{1}{\sqrt{\mu}}$. Hence, for $\displaystyle r\ge \rho(x)$ and $B(x_0,r)=[x_0-r, x_0+r]$, $\displaystyle\frac{1}{\abs{B(x_0,r)}}\int_{B(x_0,r)}\abs{f(x)}~dx\le\frac{1}{2r}\int_{B(0,1)}\abs{f(x)}~dx\le C\sqrt{\mu}$. So $f\in BMO_{\L_\mu}$. Now,
\begin{align*}
t\partial_t \P_tf(0)
    &=C\int_0^\infty t\left(1-\frac{t^2}{2s}\right)\frac{e^{-t^2/(4s)}}{s^{3/2}} \int_{|y|<1}\frac{e^{-y^2/(4s)}}{s^{1/2}}(-\log\abs{y})~dy~e^{-s\mu}~ds \\
    &= C\int_0^\infty w^2\left(1-w^2\right)e^{-w^2/2}\int_{|zt|<1}\frac{e^{-(zw)^2/2}}{s^{1/2}}~(-\log\abs{zt})~dz~ e^{-\frac{t^2}{2w^2}\mu}~\frac{dw}{w} \\
    &= C\int_0^\infty w\left(1-w^2\right)e^{-w^2/2}\int_{|zt|<1}e^{-(zw)^2/2}(-\log\abs{z})~dz~e^{-\frac{t^2}{2w^2}\mu}~dw \\
    &\quad +C\int_0^\infty w\left(1-w^2\right)e^{-w^2/2}\int_{|zt|<1}e^{-(zw)^2/2}(-\log\abs{t})~dz~e^{-\frac{t^2}{2w^2}\mu}~dw=: I+II.
\end{align*}
Observe that
\begin{align*}
    \abs{I} &\leq C\int_0^\infty we^{-w^2/c}\int_\Real e^{-(zw)^2/2}\abs{\log|z|}~dz~dw \\
     &\le C\int_0^\infty w e^{-w^2/c}\left(\int_{|z|<1}(-\log|z|)~dz+\int_{|z|>1}e^{-(zw)^2/2}|z|^\delta~dz\right)~dw \\
     & \le C\int_0^\infty we^{-w^2/c}\left(1+\frac{1}{w^\delta}\right)~dw\le C,
\end{align*}
where $\delta <1.$ For the second integral,
$$\abs{II}\leq C\abs{\log|t|}\int_0^\infty we^{-w^2/c}\int_\Real e^{-(zw)^2/2}~dz~dw= C\abs{\log|t|}\int_0^\infty e^{-w^2/c}~dw=C\abs{\log|t|}.$$
Therefore the two integrals that define $t\partial_t\P_tf(0)$ are (absolutely) convergent. The limit when $t\to0$ of
 the second term $II$ above is infinity. Thus $t\partial_t\P_tf(0)\to\infty$ as $t\to 0$.



\end{document}